\documentclass[12pt,reqno]{amsart}
\usepackage{a4wide,ulem}
\usepackage{amsmath,amsthm,amsfonts,amssymb,amscd, multicol, verbatim, enumerate}
\usepackage{verbatim}
\usepackage{fullpage}
\input{xy}
\usepackage{color}
\usepackage{graphicx}
\graphicspath{ {./images/} }
\usepackage[english]{babel}
\usepackage{cite}
\usepackage{amsfonts}
\usepackage{latexsym}
\usepackage{amsthm}
\usepackage{mathtools}
\usepackage{amsopn}
\usepackage{amsmath}
\usepackage[all]{xy}
\usepackage{verbatim}
\usepackage{amssymb}
\usepackage{mathrsfs}
\usepackage{float}

\usepackage{fullpage}

\usepackage{amscd}
\usepackage{amscd}
\usepackage{color}

\usepackage{hyperref} 
\usepackage{bbm}

\usepackage{graphicx}

\usepackage{tikz-cd}

\usepackage{ulem}
\newcommand{\trid}{\triangleright}
\newcommand{\id}{\mathrm{id}}

\newcommand{\BB}{\mathcal{B}}
\newcommand{\CC}{\mathbb{C}}
\newcommand{\CB}{{\CCC\BB}}
\newcommand{\CCC}{\mathcal{C}}
\newcommand{\NN}{\mathbb{N}}

\newcommand{\Hom}{\mathrm{Hom}}

\newcommand{\GL}{\mathrm{GL}}

\newcommand{\SSS}{\mathcal{S}}

\newcommand{\Vu}{\mathcal{V}}

\newcommand{\Sn}{\mathbb{S}_n}

\newcommand{\sym}{\mathrm{Sym}}
\newcommand{\FP}{\mathcal{FP}}
\newcommand{\PB}{\mathcal{CB}}

\newcommand{\ydG}{\prescript{{\Bbbk} G}{{\Bbbk} G}{\mathcal{YD}}}
\newcommand{\ydH}{\prescript{H}{H}{\mathcal{YD}}}
\newcommand{\ydHs}{\prescript{H_{\sigma}}{H_{\sigma}}{\mathcal{YD}}}

\newtheorem{theorem}{Theorem}[section]
\newtheorem{lemma}[theorem]{Lemma}
\newtheorem{corollary}[theorem]{Corollary}
\newtheorem{proposition}[theorem]{Proposition}
\newtheorem{definition}[theorem]{Definition}
\newtheorem{remark}[theorem]{Remark}
\newtheorem{example}[theorem]{Example}

\author{Giovanna Carnovale\textsuperscript{a}} \address{\textsuperscript{a}Dipartimento di Matematica ``Tullio Levi-Civita'', Università degli Studi di Padova, via Trieste 63, 35121 Padova, Italy} \email{carnoval@math.unipd.it} \author{Francesco Esposito\textsuperscript{a}} \address{\textsuperscript{a}Dipartimento di Matematica ``Tullio Levi-Civita'', Università degli Studi di Padova, via Trieste 63, 35121 Padova, Italy} \email{esposito@math.unipd.it} \author{Lleonard Rubio y Degrassi\textsuperscript{a,*}} \address{\textsuperscript{a}Dipartimento di Matematica ``Tullio Levi-Civita'', Università degli Studi di Padova, via Trieste 63, 35121 Padova, Italy} \email{lleonard@math.unipd.it} \thanks{\textsuperscript{*}Corresponding author: Lleonard Rubio y Degrassi, \href{mailto:lleonard@math.unipd.it} {lleonard@math.unipd.it}.}
\title{Approximation of graded bialgebras}

\begin{document}
\begin{abstract}
Motivated by an equivalence of categories established by Kapranov and Schechtman, we introduce, for each non-negative integer $d$, the category of connected bialgebras modulo $d+1$. We show that these categories fit into an inverse system of categories  whose inverse limit category is equivalent to the category of connected bialgebras. In addition, we extend the notion of approximation  of connected  bialgebras to those that are not necessarily generated in degree 1 and show that, for connected bialgebras in the category of Yetter-Drinfeld modules over a Hopf algebra, approximation is compatible with cocycle twisting. 
\end{abstract}
\maketitle

\section{Introduction}

Connected bialgebras over a base field $\Bbbk$ are $\mathbb N$-graded bialgebras, such that the degree $0$ component is $\Bbbk$. Notable examples are: the tensor algebra, the cotensor (big shuffle) algebra and the symmetric algebra. This notion can be generalized if we replace the category of vector spaces by any additive $\Bbbk$-linear braided monoidal category $\Vu$: in this case key examples include the exterior algebra, cohomology algebras of simply-connected Lie groups, the positive part of quantized enveloping algebras, and, more generally, Nichols (shuffle) algebras.  Since our main example of such a braided category is the category $\Vu= {}_{H}^{H}\mathcal{YD}$ of Yetter-Drinfeld modules over a $\Bbbk$-Hopf algebra $H$, we concentrate on $\Bbbk$-linear categories  $\Vu$ whose objects are  $\Bbbk$-vector spaces with possibly additional structure. In other words, we require that the category $\Vu$ has a well-behaved functor to the category of $\Bbbk$-vector spaces. Most of the results we obtain could be formulated in a more general setting, to the cost of readability. 

 It has been recently shown by Kapranov and Schechtman in \cite{KS3} that the category $\CB(\Vu)$ of connected bialgebras in $\Vu$  has a geometric counterpart: it is equivalent to the category $\FP(\sym(\mathbb{C}); \Vu)$ of factorized perverse sheaves on the space $\sym(\mathbb C)$ of monic polynomials and with values in $\Vu$. One of our motivations is the translation of basic algebraic properties  of connected bialgebras into geometric statements by means of this equivalence. 
 
As one can see from the examples  mentioned above, connected bialgebras may have very different nature: for instance, the symmetric and exterior algebra are quadratic, that is, they have a presentation with generators in degree $1$ and homogeneous relations in degree $2$; the positive part of the quantized enveloping algebra can be generated by degree $1$ elements with relations up to a degree $d>2$ that depends on the Cartan matrix. Even for the well-studied class of Nichols algebras, there is no general method to determine a minimal set of homogeneous relations  or to give a bound to the degree of such relations. It is then natural to investigate whether, for a connected bialgebra $A$, there is a $d$ such that $A$ has a (graded) presentation whose relations are  bounded in degree $d$, and to estimate such a $d$.  If the number of generators is finite, existence of such a bound is equivalent to  $A$ being finitely presented. 

When a graded algebra $A$ is generated in degree one, then one has a classical notion of 
$d$-atic cover, by considering the covering algebra whose relations are generated up to degree $d$. However, connected bialgebras such as the cotensor algebra or the cohomology ring of  $SU(n)$, see Example \ref{ex:cotensor} and Example \ref{Ex:Liegroup}, are not generated in degree one, and a more general approach is required to deal with them. 

The first main goal of this paper is to generalize this construction  to all connected bialgebras. In order to do so, we introduce the category of connected bialgebras modulo $d+1$ in $\Vu$, 
along with the $d$-th truncation and the $d$-th extension functors, and define the $d$-th approximation functor
as a composition of the truncation and the extension functors. Our construction specializes to the classical procedure  for bialgebras generated in degree one,
(Theorem \ref{Thm1}). Along the way, we  prove that the $d$-th extension functor is left adjoint to the $d$-th truncation functor,  see Proposition \ref{prop:adjunction}. 

In addition, we show in Proposition \ref{prop:limit-PB} that when $d$ runs in $\NN$, the family of  connected bialgebras modulo $d+1$ in $\Vu$
fits into a system of categories whose inverse limit category, in the spirit of \cite[\S 5.1]{Schn}, is  equivalent to $\CB(\Vu)$. 

The second main goal of this work is to show that the $d$-th approximation functors are well-behaved with respect to equivalence of categories, (Theorem \ref{thm:equivaprox}). In particular, if $\Vu= {}_{H}^{H}\mathcal{YD}$ for a Hopf algebra $H$ with invertible antipode and $\sigma$ is a left cocycle on $H$, then we exploit results in \cite{MO,AFGV} to show that the $d$-th approximation functors are well behaved with the equivalence induced by twisting $H$ by $\sigma$, see Theorem \ref{thm:twistapprox}. 
As a consequence, if two connected  bialgebras  $A$, $A'$ are twist equivalent, then the same holds for their respective $d$-th approximations, and  $A$ coincides with its $d$-th approximation if and only if $A'$ coincides with its $d$-th approximation.

This in particular holds for Nichols algebras, Corollary \ref{cor:d-twist}. In the Appendix \ref{Appendix}, we provide an alternative proof of Corollary \ref{cor:d-twist} that does not require categorical machinery. 

In   the spirit of \cite{KS3}, in  \cite{CERyD}   we  develop the geometric counterpart of the truncation and extension functors. We expect that the equivalence introduced by Kapranov and Schechtman can be generalized to an equivalence between the category of connected bialgebras modulo $d+1$ in $\Vu$ and the category of factorized perverse sheaves on suitable open subspaces of $\sym(\CC)$. This would allow the translation of algebraic statements about finite presentation and approximations of bialgebras into statements about perverse sheaves on $\sym(\CC)$. This is our main motivation behind the introduction of the approximation functors.

\section{Background}

\subsection{Bialgebras in braided categories}
\label{S:BMC}
 
With the purpose of fixing notation, we recall the definitions of monoidal category and braiding. The reader is referred to \cite{EGNO} for a complete treatment.

A monoidal category structure on a category $\mathcal{C}$ is the datum of a bifunctor 
    $\otimes\colon \mathcal{C}\times\mathcal{C}
    \rightarrow \mathcal{C}$, called tensor product, a natural isomorphism $a\colon(-\otimes -)\otimes\to -\otimes(-\otimes -)$, called the associativity constraint, an object $\mathbf{1}_{\mathcal{C}}$ of $\mathcal{C}$ called the unit object, and an isomorphism $\upsilon\colon \mathbf{1}_{\mathcal{C}}\otimes \mathbf{1}_{\mathcal{C}}\to \mathbf{1}_{\mathcal{C}}$, called unit constraint,  subject to the pentagon and unit compatibility axioms \cite[(2.2), (2.3)]{EGNO}.  Notable examples are the category of $\Bbbk$-vector spaces, the category of representations of a Hopf algebra, the category of representations of a Lie algebra.  A monoidal category is said to be strict if for all objects $A,B,C$ in ${\mathcal C}$ there are equalities $(A\otimes B)\otimes C=A\otimes(B\otimes C)$ and $A\otimes \mathbf{1}_{\mathcal{C}}=A=\mathbf{1}_{\mathcal{C}}\otimes A$ and the associativity and unit constraints are the identity maps. An additive  monoidal category is a category that is both additive and monoidal and such that the tensor product is an additive functor in each variable.
    
    Let $\mathcal{C}, \mathcal{C}'$ be 
    two monoidal categories with tensor products $\otimes$ and $\tilde{\otimes}$, respectively. A 
    monoidal functor from $\mathcal{C}$ to $\mathcal{C}'$ is a pair $(\mathcal{F}, \omega)$, where $\mathcal{F}\colon \mathcal{C}\to 
    \mathcal{C}'$ is a functor  such that $\mathcal{F}
  (\mathbf{1}_{\mathcal{C}})$ is isomorphic to 
  $\mathbf{1}_{\mathcal{C}'}$, and $\omega_{A,B}\colon \mathcal{F}(A)\tilde{\otimes} \mathcal{F}(B) \to 
  \mathcal{F}(A\otimes B)$ is a natural isomorphism for all objects $A,B\in \mathcal{C}$
   such that $(\mathcal{F}, \omega)$ satisfies the monoidal structure axiom \cite[Equation 2.23]{EGNO}.  A monoidal functor $(\mathcal{F},\omega)$ is {strict} if $\omega_{A,B}$ and the isomorphism from $(\mathbf{1}_{\mathcal{C}})$  to 
  $\mathbf{1}_{\mathcal{C}'}$ are identity maps for all objects $A,B\in \mathcal{C}$. A monoidal functor $(\mathcal{F},\omega)$ is a {monoidal equivalence} if $\mathcal{F}$  is an equivalence of the underlying categories. 

Given a monoidal category structure  $(\otimes,\mathbf{1}_{\mathcal{C}},a,\upsilon)$ on ${\mathcal C}$, the {opposite}  monoidal category structure 
on $\mathcal{C}$ is given by the bifunctor $\otimes^{op}$ such that $A \otimes^{op}B:=B\otimes A$ for any pair of objects $A,B$ in ${\mathcal C}$,  the associator $a^{op}_{A,B,C}=a^{-1}_{C,B,A}$, for any triple of objects $A,B,C$ in ${\mathcal C}$ together with $\mathbf{1}_{\mathcal{C}}$, and $\upsilon$. 
By Mac Lane's theorem, every monoidal category is monoidally equivalent to a strict one \cite[Theorem 2.8.5]{EGNO}, hence we will treat all monoidal categories as strict ones.

A {braiding} on a (strict) monoidal structure $(\mathcal{C},\otimes,\mathbf{1}_{\mathcal{C}})$ is an isomorphism of functors $c\colon\otimes \rightarrow \otimes^{op}$ satisfying the following equations on triples of objects $A,B$ and $C$: 
\begin{align}\label{eq:braiding}
    c_{A,B\otimes C}&=(\id_B \otimes c_{A,C})\circ(c_{A,B}\otimes \id_C),\\
  \nonumber c_{A\otimes B,C}&=(c_{A,C}\otimes \id_B)\circ(\id_A\otimes c_{B,C})  
\end{align}

Equations \eqref{eq:braiding} imply the {Yang-Baxter equation} on ${\mathcal C}\times {\mathcal C}\times {\mathcal C}$:
\begin{equation}\label{eq:braid1}    
(\id_C\otimes c_{A,B}) \circ (c_{A,C}\otimes \id_B) \circ (\id_A \otimes c_{B,C}) =
    (c_{B,C}\otimes \id_A) \circ (\id_B \otimes c_{A,C}) \circ (c_{A,B}\otimes \id_C). \end{equation}

A monoidal category with a braiding is called a {braided category}.   Let $\mathcal{C}, \mathcal{C}'$ be two braided categories having tensor products $\otimes, \tilde{\otimes}$ and braidings  $c, c'$, respectively. A monoidal functor $(\mathcal{F},\omega)$ from $\mathcal{C}$ to $\mathcal{C}'$ is {braided} if $ \mathcal{F}(c_{A,B}) \circ \omega_{A,B} = \omega_{B,A}\circ c'_{\mathcal{F}(A), \mathcal{F}(B)}$ for all objects $A, B$ in $\mathcal{C}.$

A {braided monoidal equivalence} of braided monoidal categories is a braided monoidal functor $(\mathcal{F},\omega)$ such that $\mathcal F$ is an equivalence of the underlying categories.  
\medskip

The example of  braided  monoidal category we are mainly interested in is the category $\ydH$ of Yetter-Drinfeld modules over a Hopf algebra $H$ with invertible antipode.  In order to give a concrete example we recall that, in the particular case in which $H=\Bbbk G$ is the group algebra of a finite group, an object in $\ydG$ 
is a vector space $V$ together with an action of $G$ and a $G$-grading $V=\bigoplus_{g\in G}V_g$ satisfying $g\cdot V_h=V_{ghg^{-1}}$ for $g,\,h\in G$. The tensor product of objects in $\ydG$ is naturally equipped with a Yetter-Drinfeld module structure and the braiding on $V\otimes W$ is given by $c_{V,W}(v\otimes w)=h\cdot w\otimes v$ for every $v\in V_h,\, h\in G$ and  $w\in W$.  
We refer to \cite[\S 3.4]{HS} for the  definition of Yetter-Drinfeld modules and braiding in $\ydH$ for general $H$.

\medskip

With this main example in mind, from now on  $\Vu$ will always denote an abelian, $\Bbbk$-linear braided category with a strict additive monoidal functor to the category of $\Bbbk$-vector spaces.  Thus, we will regard its objects as vector spaces without further notice. 

\medskip

As in the category of vector spaces, one defines the notions of an algebra and a coalgebra  in $\Vu$. An {algebra} in $\Vu$ is an object $A$ in $\Vu$ together  with an associative morphism $m_A\colon A\otimes A\to A$ and a morphism $u_A\colon \mathbf{1}_\Vu\to A$ satisfying the relation of a unit morphism on $A$, \cite[(3.1.3), (3.1.4)]{HS}. Dually, a {coalgebra} in $\Vu$ is an object $C$ in $\Vu$, together with a coassociative morphism $\Delta_C\colon C\to C\otimes C$ and a morphism $\epsilon_C\colon C \to \mathbf{1}_\Vu $ satisfying a counit diagram, \cite[\S 3.1]{A2}. Algebra, respectively, coalgebra morphisms are morphisms in $\Vu$ that are compatible with the operation and the unit, respectively the counit.

\medskip

Given two algebras (respectively coalgebras) $A,B$ in ${\Vu}$, then $A\otimes B$ is an algebra (respectively a coalgebra) in ${\Vu}$ where the multiplication map (resp. comultiplication map) is defined as: $m_{A\otimes B}= (m_A\otimes m_B)\circ (\id\otimes c_{B,A}\otimes \id)$, $\Delta_{A\otimes B}= (\id\otimes  c_{A,B}\otimes \id)\circ  (\Delta_A\otimes \Delta_B)$.  This allows  to introduce the notion of a bialgebra in $\Vu$: assume that $(A, m_A, u_A)$ is an algebra and $(A, \Delta_A, \epsilon_A)$ a coalgebra in $\Vu$. Then $(A, m_A, u_A, \Delta_A, \epsilon_A)$ is a {bialgebra} in $\Vu$ if $\Delta_A$ and $\epsilon_A$ are algebra morphisms in $\Vu$.

We denote by $m_{A}^{(n)}\colon A^{\otimes (n+1)}\to A$ the $n$-fold iteration of the multiplication, that is,  $m_{A}^{(1)}:=m_{A}$ and for $n>1$ we define $m_{A}^{(n)}:= m_{A}^{(n-1)}\circ(m_A\otimes \id^{(n-1)})$.
If $f\colon A\to B$ is an algebra morphism, then $m_B^{(n)}\circ f^{\otimes (n+1)}=f\circ m_A^{(n)}$ for every $n\geq1$.

\medskip

Let ${\mathbb N}$ be the set of non-negative integers. Let ${\mathcal V}_{gr}$ be the category of ${\mathbb N}$-graded objects in $\Vu$, that is, objects are of the form $A=\bigoplus_{j\geq0} A_j$ where each $A_j$ is an object in $\Vu$ and morphisms are homogeneous.  
It is not hard to verify that ${\mathcal V}_{gr}$ is again braided monoidal, \cite[Definition 2.1]{Sch}. In more detail, for two objects $A=\bigoplus_{j\geq 0}A_j$ and $B=\bigoplus_{j\geq0}B_j$ in 
${\mathcal V}_{gr}$, we define {$A\otimes B$ where}  
$(A\otimes B)_n=\bigoplus_{i=0}^n A_i\otimes B_{n-
i}$;   the unit in ${\mathcal V}_{gr}$ 
is given by the unit $\mathbf 1_\Vu$ in 
${\Vu}$ considered as an element concentrated in degree $0$. The  braiding in ${\mathcal V}_{gr}$  is given by the sequence 
$c^{gr}_{A,B}=(\bigoplus_{i+j=k}c_{A_i,B_j})_{k\geq 0}$, where $c_{A_i,B_j}$ is the braiding in $\Vu$.

\medskip

We denote by $\CB(\Vu)$ the category   of  {connected bialgebras}, called  {primitive bialgebras} in \cite{KS3}. An object $A=\bigoplus_{n\geq0}A_n$ in $\CB(\Vu)$  is a bialgebra in $\Vu_{gr}$, that is, it is  ${\mathbb N}$-graded  as an algebra   and a  coalgebra, satisfying the additional connectedness condition  $A_0=\mathbf{1}_\Vu$.

\begin{example}\label{ex:tensor-cotensor-Nichols}
 {\rm    Let $V$ be an object in $\Vu$. Setting $T^m(V):=V^{\otimes m}$ for each $m\geq 0$, the object $T(V)=\bigoplus_{m\geq0}T^m(V)$ of $\Vu_{gr}$
can be equipped with two natural connected bialgebra structures: the tensor algebra $T_{!}(V)$, with the usual tensor multiplication and comultiplication determined requiring that all elements in $V$ are primitive; and the cotensor algebra $T_{\ast}(V)$, with the multiplication given by the braided shuffle product \cite[Section 6.5]{HS} and comultiplication given by deconcatenation. When $\Vu$ is the category of vector spaces with usual flip, the shuffle algebra structure was introduced in \cite{Ree}.  
One defines  the Nichols algebra $\mathcal{B}(V)$ as the image of the unique graded
algebra morphism $\Omega\colon  T_{!}(V)\mapsto 
T_{\ast}(V)$ which extends the identity on $V$. This morphism is also a coalgebra morphism hence  $\mathcal{B}
(V)$ is an object in $\mathcal{CB}(\Vu)$. } 
\end{example}

\begin{remark}\label{rmk:bmbmf}
   {\rm Let $\Vu, \Vu'$ be abelian, $\Bbbk$-linear braided monoidal categories  and let $(\mathcal{T},\omega)$ be a braided monoidal functor from $\Vu$ to $\Vu'$. Since $
\mathcal{T}$ preserves bialgebras \cite[Remark 3.2.13]{HS},  then $(\mathcal{T},\omega)$  induces a functor from the category of bialgebras in $\Vu$ to the category of bialgebras in $\Vu'$. If, in addition, $\mathcal T$  is additive, then $(\mathcal{T},\omega)$ induces a monoidal functor  $(\overline{\mathcal{T}}, \overline{\omega})$ from $\Vu_{gr}$ to 
$\Vu'_{gr}$. More precisely, let   
$A=\bigoplus_{i\geq 0}A_i$ and $ B=\bigoplus_{j\geq 0}B_j$  be objects in $\Vu_{gr}$. We define 
$\overline{\mathcal{T}}(A):=\bigoplus_{i\geq 0}\mathcal{T}
(A_i)$ and for every morphism $f=(f_i)_{i\geq 
0}\colon A\to B$, we define $\overline{\mathcal{T}}(f):=
(\mathcal{T}(f_i))_{i\geq 0}$. Similarly, 
$\overline{\omega}_{A,B}\colon \overline{\mathcal{T}}(A)\tilde{\otimes} \overline{\mathcal{T}}(B)\to \overline{\mathcal{T}}(A\otimes B)
$ is a natural  isomorphism defined by $\overline{\omega}_{A,B}:=\bigoplus_{i,j\geq 0}\omega_{A_i,B_j}$. Moreover, $(\overline{\mathcal{T}}, \overline{\omega})$ is a braided  functor since $(\mathcal{T},\omega)$ is a braided functor. Therefore  $(\overline{\mathcal{T}},  \overline{\omega})$ induces a functor,  which by abuse of notation is again denoted by $\mathcal{T}$, from $\CB(\Vu)$ to $\CB(\Vu')$. For $A\in \CB(\Vu)$ with multiplication $m_A$ and comultiplication $\Delta_A$, the multiplication and comultiplication on $\mathcal T(A)$ are defined, respectively, by $\mathcal T(m_A)\circ \overline{\omega}_{A,A}$ and 
$\overline{\omega}_{A,A}^{-1}\circ \mathcal T(\Delta_A)$.}
\end{remark}

\begin{remark}\label{notbialg}
 {\rm If $\Vu$ is not symmetric, that is, if $c\circ c\not=\id_{\otimes}$, then the tensor product of  two bialgebras in ${\Vu}$ is not necessarily a bialgebra.  For example, let $G$ be a group and let $V_1,V_2$ be objects in the braided category ${}_{\Bbbk G}^{\Bbbk G}\mathcal{YD}$ of Yetter-Drinfeld modules over $\Bbbk G$.  Then $\mathcal{B}(V_1)\otimes \mathcal{B}(V_2)$ is not always a bialgebra in ${}_{\Bbbk G}^{\Bbbk G}\mathcal{YD}$, see \cite[Proposition 1.10.12, Corollary 1.6.10]{HS}. Hence, $\CB(\Vu)$ is not equipped with a monoidal structure if we consider the tensor product of algebras and coalgebras. It is also not $\Bbbk$-linear, nor additive. }
\end{remark}

Let $\PB^1(\Vu)$ be the full subcategory of $\PB(\Vu)$ whose objects are generated as algebras by their degree $1$ component. An object $B=\bigoplus_{j\geq0}B_j$ in $\PB^1(\Vu)$ is then the quotient of $T_!(B_1)$ by some graded ideal and coideal $J$ contained in $\left(\bigoplus_{m\geq 2}T_!^2(V)\right)$. 
For any $d\in \NN$, the subspace $\bigoplus_{m=2}^d(J\cap T_!^m(B_1))$ is a (possibly trivial) graded coideal of $T_!(V)$. Hence, the generated ideal $\left(\bigoplus_{m=2}^d(J\cap T_!^m(B_1))\right)$ is a graded ideal and coideal.  The bialgebra quotient \begin{equation}
\label{eq:approximation}
    \widehat{B}_d:=T_!(B_1)/\left(\bigoplus_{m=2}^d(J\cap T_!^m(B_1))\right)
\end{equation} is again an object in $\CB^1(\Vu)$. A morphism $f\colon T_!(B_1)/J\to T_!(B'_1)/J'$ in $\PB^1(\Vu)$ is completely determined by its degree $1$ component $f_1\colon B_1\to B_1'$, which is necessarily a coalgebra morphism. Composing $f_1$ with the natural inclusion $B_1'\subset T_!(B_1')$ followed by the projection to $\widehat{B'}_d$ gives a coalgebra morphism in $\Vu_{gr}$. The latter can be extended to a unique bialgebra morphism $T_!(B_1)\to \widehat{B'}_d$ which factors through $\widehat{B}_d$. We denote the induced morphism by $\widehat{f}_d$. In this way we have defined an endofunctor $\PB^1(\Vu)\to \PB^1(\Vu)$ called the $d$-atic covering functor. 
For $d=2$, the algebra $\widehat{B}_d$ is the usual quadratic cover $T_!(B_1)/(J\cap T^2_!(B_1))$.
\medskip 

However, not all connected bialgebras are generated in degree one.
\begin{example}\label{ex:cotensor}
{\rm   Let $V$ be a vector space.  Then, the big shuffle algebra $T_*(V)$ as in \cite{Ree} is not generated in degree $1$ as an algebra: indeed, for linearly independent $v,w\in V$, the element $v\otimes w$ is not symmetric, hence it is not generated by elements in $T^1(V)=V$.} 
\end{example}

If $A$ is the cohomology of a compact connected Lie group $G$, then $A$ is a Hopf algebra in the category of $\mathbb{N}$-graded vector spaces with braiding given by $c(v\otimes w)=(-1)^{pq} w\otimes v$, where $v$ and $w$ have degree $p, q$, respectively, see \cite{MM} for the statement in terms of homology. In addition,  $A$  is isomorphic (as an algebra) to an exterior algebra on odd-dimensional generators \cite{Hopf}. These  bialgebras give further examples of connected bialgebras that are not generated in degree one. 

\begin{example}\cite[Example 5.F ]{Mc}\cite[\S 6 Corollary 6.5]{MimTod}
\label{Ex:Liegroup}
{\rm Assume that the characteristic of $\Bbbk$ is $0$. Then we have the following algebra isomorphisms:
 \begin{itemize}
     \item $\mathrm{H}^*(SU(n);\Bbbk)\simeq \Lambda[x_3,x_5,\dots, x_{2n-1}]$;  
     \item $\mathrm{H}^*(SO(2m+1;\Bbbk))\simeq \Lambda [x_3, x_7, \dots, x_{4m-1}]$,
 \end{itemize}
 where the subscripts on the generators refer to their grading in the cohomology ring, that is, $x_i \in H^i(G;\Bbbk)$.}
 \end{example}

\subsection{Duality}\label{subsec:dual} Let $V$ be an object of a braided monoidal category $\mathcal{C}$. A {left dual} of $V$ is a triple $(V^*, \mathrm{ev}_V, \mathrm{coev}_V)$, where $V^*$ is an 
object of $\mathcal{C}$,  $\mathrm{ev}_V\colon V^*\otimes V \to \mathbf{1}_{\mathcal{C}}$ and $\mathrm{coev}_V\colon  \mathbf{1}_{\mathcal{C}}\to V\otimes V^*$ are morphisms such that $(\mathrm{ev}_V\otimes \id)\circ  (\id\otimes \mathrm{coev}_V)=\id_{V^*}$, $ (\id\otimes \mathrm{ev}_V) \circ (\mathrm{coev}_V 
\otimes\id)=\id_{V}$. 

A braided strict monoidal category $\mathcal C$ is called {rigid} if each object has a left dual.
The prototypical examples of rigid braided monoidal categories are the category of finite-dimensional vector spaces and the category of  finite-dimensional Yetter-Drinfeld modules over a Hopf algebra \cite[Remark 4.2.4]{HS}. If $\mathcal C$ is rigid, then for any morphism $f\colon V \to W$ one defines the left dual $f^*\colon  W^*\to V^*$ as $f^*=(\mathrm{ev}_W\otimes\id_{V^*})(\id_{W^*}\otimes f\otimes\id_{V^*})(\id_{W^*}\otimes\mathrm{coev}_V)$, see \cite[Remark 3.5.2]{HS}.  The contravariant functor $(-)^*\colon \mathcal C \to\mathcal C$, which sends an object $V$ to $V^*$ and maps a morphism $f$ to $f^*$ is called the {left duality functor}. By \cite[Theorem 3.5.8]{HS} there are natural isomorphisms 
$$\varphi_{V,W}\colon  V^*\otimes W^*\to (V\otimes W)^*,$$  for any $V$ and $W$ in $\mathcal C$, such that $((\ )^*,\varphi)$ becomes a braided monoidal equivalence. 
In addition, there is a natural isomorphism from the identity functor to $(-)^{**}$.

\medskip

More generally, a {duality} in a category $\mathcal{A}$ is  a contravariant functor $\mathbb {D}\colon  \mathcal{A} \to \mathcal{A}$ with a quasi-inverse that is naturally isomorphic to $\mathbb D$. 

\medskip

Assume that $\Vu$, with braiding $c=(c_{V,W})$ is also rigid. Since we assume $\Vu$ to be an additive monoidal category, the left duality functor is additive.  Taking the graded dual induces an additive duality functor $\mathbb D_{gr}\colon \Vu_{gr}\to \Vu_{gr}$. This is defined on any object $V=\bigoplus_{j\geq0} V_j$ in ${\mathcal V}_{gr}$ as $\mathbb D_{gr}(V):=\bigoplus_{j\geq0} V^*_j$ and on any morphism $f\colon  V\to W$ as the morphism $\mathbb D_{gr}f\colon \mathbb D_{gr}(W)\to \mathbb D_{gr}(V)$ given by $(f^*_j)_{j\geq 0}$. 

\bigskip

Using additivity of the left duality functor $(\ )^*$  and the monoidal structure $\varphi$, we  equip $\mathbb D_{gr}$ with a monoidal structure $\varphi^{gr}$ as follows. 
Let $V=\bigoplus_{l\geq0}V_l$ and  $W=\bigoplus_{l\geq0}W_l$ be objects in $\Vu_{gr}$.  
For any $l\geq0$ we define the morphisms in $\Vu$ 
\begin{align*}
&\varphi^l_{V,W}\colon (V^*\otimes W^*)_l\longrightarrow (V\otimes W)^*_l,\\
&\varphi_{V,W}^l:=\bigoplus_{j=0}^l\varphi_{V_j,W_{l-j}}\colon \oplus_{j=0}^l(V_j^*\otimes W_{l-j}^*)\longrightarrow \oplus_{j=0}^l (V_j\otimes W_{l-j})^*\equiv\left(\oplus_{j=0}^l V_j\otimes W_{l-j}\right)^*
\end{align*}
which combine to give a natural morphism $\varphi^{gr}_{V,W}=\bigoplus_{l\geq0}\varphi^l_{V,W}$  in $\Vu_{gr}$.  Property \cite[(3.1.7)]{HS} for the collection of morphisms $\varphi_{V,W}^{gr}$ follows from the analogous property for $\varphi_{V,W}$ and additivity of the left duality functor.

\medskip

Therefore, if $A$ is an algebra  in $\Vu_{gr}$ with multiplication $m_A$ and unit $u_A$, then $\mathbb D_{gr}(A)$ is a coalgebra with comultiplication $\Delta_{\mathbb D_{gr}(A)}:=(\varphi^{gr}_{A,A})^{-1}\circ \mathbb D_{gr}(m_A)$ and counit $\mathbb D_{gr}(u_A)$. Dually, if $C$ is a coalgebra  in $\Vu_{gr}$ with comultiplication $\Delta_C$ and counit $\epsilon_C$, then $\mathbb D_{gr}(C)$ is an algebra with multiplication $m_{\mathbb D_{gr}(C)}:=\mathbb D_{gr}(\Delta_C)\circ (\varphi^{gr}_{C,C})$ and unit $u_{\mathbb D_{gr}(C)}:=\mathbb D_{gr}(\epsilon_C)$, \cite[p. 260]{scha}. If $A$, respectively, $C$ is connected, then $\mathbb D_{gr}(A)$, respectively, $\mathbb D_{gr}(C)$ is connected.

\medskip

In addition, under the above assumptions on $\Vu$, the functor $(\mathbb D_{gr},\varphi^{gr})$ is braided monoidal. This follows from the additivity of the functor $(\ )^*$ and  the equality $(c_{W_j,V_l})^*\circ \varphi_{V_l,W_j}=\varphi_{W_j,V_l}\circ c_{V_l^*,W_j^*}$ for all pairs of objects $V_l,W_j$ in $\Vu$, \cite[Theorem 3.5.8]{HS}. 
Then, $\mathbb D_{gr}$ induces a duality on $\PB(\Vu)$. Indeed, if $A$ in as object in $\PB(\Vu)$ with multiplication $m_A$, unit $u_A$, comultiplication $\Delta_A$, and counit $\epsilon_A$, then $\mathbb D_{gr}(A)$ with multiplication $m_{\mathbb D_{gr}(A)}$, unit $u_{\mathbb D_{gr}(A)}$, comultiplication $\Delta_{\mathbb D_{gr}(A)}$, and counit $\epsilon_{\mathbb D_{gr}(A)}$ is again an object in $\PB(\Vu)$, and morphisms are mapped to morphisms, \cite[Remark 3.2.13]{HS}.

\medskip

The same argument shows that any additive, braided, monoidal duality in $\Vu$ induces a duality on $\PB(\Vu).$

\medskip

Let ${\mathcal A}$ and  ${\mathcal B}$ be categories with  dualities  $\mathbb {D}, \mathbb {D}'$, respectively and let $G\colon {\mathcal A}\to {\mathcal B}$ be a  functor. Then the {dual functor} $\overline{{\mathbb D}}$
of $G$ is a functor  from ${\mathcal A}$ to $ {\mathcal B}$ defined on objects as $\overline{{\mathbb D}}G(A):={\mathbb D'}(G({\mathbb D}A))$ and on morphisms 
$f\colon A\to B$ 
as $\overline{{\mathbb D}}G(f):={\mathbb D'}(G({\mathbb D}f))$.

\section{Approximation functor}

The goal of this section is to define the $d$-th approximation functor $F_d$ on the category of connected bialgebras, as a composition of two adjoint functors: the truncation functor $\Theta_d$ and the extension functor $\Psi_{d!}$. We first need to introduce a truncated version of  $\CB(\Vu)$.

\subsection{Connected bialgebras modulo d+1.}
\label{cgbia}

Let $d\in\mathbb N$. For an object $V=\bigoplus_{j\geq0}V_j$ in $\Vu_{gr}$ its $d$-th truncation is the object $V_{\leq d}:=\bigoplus_{j\leq d}V_d$ in $\Vu_{gr}$. Similarly, for a morphism $f=(f_j)_{j\in\NN}\colon A\to A'$ in $\Vu_{gr}$, its truncation $f_{\leq d}$ is the morphism 
$(f_j)_{0\leq j\leq d}\colon A_{\leq d}\to A'_{\leq d}$.

\medskip

\begin{definition}A connected bialgebra mod $d+1$ is an algebra and coalgebra $A=\bigoplus_{l=0}^d A_l$ in $\Vu_{gr}$, such that $A_0=\mathbf{1}_\Vu$, and such that the multiplication $m_A$ and the comultiplication $\Delta_A$ satisfy the following compatibility relation:
 
\begin{equation}\label{eq:truncated}
\begin{tikzcd}
 (A\otimes A)_{\leq d}  \arrow[rr,"(m_A)_{\leq d}"]  \arrow[d,"(\Delta_A\otimes \Delta_A)_{\leq d}"']
&& A_{\leq d} \arrow[d,"(\Delta_A)_{\leq d}"] \\
(A\otimes A\otimes A\otimes A)_{\leq d} \arrow[dr,"(\id\otimes c_{A,A}\otimes\id)_{\leq d}"'] && (A\otimes A)_{\leq d} \\
& (A\otimes A\otimes A\otimes A)_{\leq d} \arrow[ur,"(m_A\otimes m_A)_{\leq d}"']
\end{tikzcd}
\end{equation}

The category  $\PB^{\leq d}(\Vu)$ of connected  bialgebras $\mod d+1$ 
is the category whose objects are connected bialgebras mod $d+1$ and whose morphisms are graded coalgebra maps preserving the multiplication.
\end{definition}

In other words, objects in $\PB^{\leq d}(\Vu)$ satisfy the bialgebra compatibility condition up to degree $d$.

\begin{example}{\rm 
Let $A=\bigoplus_{i=0}^dA_i\in\CB^{\leq d}(\Vu)$, with comultiplication $\Delta_A$ and multiplication $m_A$.
\begin{enumerate}
\item If $d=0$, then $\CB^{\leq 0}(\Vu)$ is the category consisting of only the unit object ${\mathbf 1}_\Vu$ and the identity morphism.
    \item If $d\geq 1$, then for any $a\in A_1$, $b\in A_{d}$ we have $\Delta(a)=a\otimes 1+1\otimes  a$ and $m_A(a\otimes b)=0$.
    \item If $d=1$, then $m_A$ is trivial on $A_1\otimes A_1$, and it is  scalar multiplication on $A_0\otimes A_1$ and $A_1\otimes A_0$. Hence the assignment $A\mapsto A_1$ determines an equivalence of categories $\CB^{\leq 1}(\Vu)\simeq \Vu$. 
    \item If $d=2$, then $\CB^{\leq 2}(\Vu)$ is equivalent to the category whose objects are quadruples $(V_1,V_2, \mu,\delta)$, where $V_1,V_2$ are objects in $\Vu$, and $\mu\colon V_1\otimes V_1\to V_2$ and $\delta\colon V_2\to V_1\otimes V_1$ are morphisms satisfying $\delta\mu(a\otimes b)=a\otimes b+c_{V_1,V_1}(a\otimes b)$ 
    for all $a,\,b\in A_1$. A morphism from $(V_1,V_2, \mu,\delta)$ to $(W_1,W_2, \mu',\delta')$ is  given by a pair $(f_1, f_2)$ of maps $f_1\colon V_1\to W_1, f_2\colon V_2\to W_2$ such that $f_2(\mu(a\otimes b))=\mu'(f_1(a)\otimes f_1(b))$ and such that $\delta'(f_2(c))=(f_1\otimes f_1)(\delta(c))$ for any  $a,b \in V_1$ and $c\in V_2$.   
\end{enumerate}}
    \end{example}
\begin{remark}\label{rmk:bmbmf2}
{\rm   Let $\Vu, \Vu'$ be abelian, $\Bbbk$-linear braided monoidal categories and let $(\mathcal{T},\omega)$ be an additive braided monoidal functor from $\Vu$ to $\Vu'$. Let $(\overline{\mathcal{T}}, \overline{\omega})$ be the induced braided monoidal functor  from $\Vu_{gr}$ to $\Vu'_{gr}$. As for connected bialgebras, $(\overline{\mathcal{T}}, \overline{\omega})$  induces a functor, denoted  by $\mathcal{T}^d$, from $\CB^{\leq d}(\Vu)$ to $\CB^{\leq d}(\Vu')$.  For an algebra $A$ in $\CB^{\leq d}(\Vu)$ with multiplication $m_A$ and comultiplication $\Delta_A$, the multiplication and comultiplication on $\mathcal T^d(A)$ are defined, respectively, by
   $\mathcal T^d(m_A)\circ \overline{\omega}_{A,A}$ and $\overline{\omega}^{-1}_{A,A}\circ\mathcal T^d(\Delta_A)$. 
   In addition, $\overline{\mathcal{T}}(A)$ satisfies the bialgebra compatibility condition up to degree $d$ because $A$ does and $\overline{\mathcal{T}}$ is a graded braided monoidal functor.}
\end{remark}
  
\begin{remark}\label{rmk:duality}{\rm As for ordinary bialgebras, the notion of bialgebra modulo $d+1$ is self-dual. Indeed, assume that $\Vu$ is rigid and let $(\mathbb D_{gr},\varphi^{gr})$ be the additive monoidal functor as in Section \ref{subsec:dual}. If $A=\bigoplus_{l=0}^d A_l$ is an object in $\PB^{\leq d}(\Vu)$, then 
the comultiplication $\Delta_A$ involves only terms in degree $\leq d$ and $m_A(A_l\otimes A_j)=0$ whenever $j+l>d$.  

Then, $\mathbb D_{gr}(A)=\bigoplus_{l=0}^dA_l^*$ is a graded algebra with multiplication $m_{\mathbb D_{gr}(A)}=\mathbb D_{gr}(\Delta_A)\circ \varphi_{A,A}^{gr}$ and a graded coalgebra with comultiplication $\Delta_{\mathbb D_{gr}(A)}=(\varphi_{A,A}^{gr})^{-1}\circ \mathbb D_{gr}(m_A)$. By construction, it satisfies $\mathbb D_{gr}(A)_0=\mathbf{1}_\Vu$ and $m_{\mathbb D_{gr}(A)}$ is the zero map outside $(\mathbb D_{gr}(A)\otimes \mathbb D_{gr}(A))_{\leq d}$. 

Truncating the natural isomorphism $\varphi^{gr}$ from Section \ref{subsec:dual} up to degree $d$ we obtain a natural isomorphism \[
\varphi^{\leq d}_{A,A}\colon  (\mathbb D_{gr}(A)\otimes \mathbb D_{gr}(A))_{\leq d}\to (\mathbb D_{gr}(A\otimes A))_{\leq d}.
\]
Making use of the naturality of $\varphi^{\leq d}$, of commutativity of the diagram \eqref{eq:truncated}, and of the fact that $(\mathbb D_{gr},\varphi^{gr})$ is braided monoidal, a diagram chasing leads to commutativity of the diagram
\[
\begin{tikzcd}%[cramped, sep=small]
 (\mathbb D_{gr}(A)\otimes \mathbb D_{gr}(A))_{\leq d}    
& (\mathbb D_{gr}(A))_{\leq d} \arrow[l, swap,"(\Delta_{\mathbb D_{gr}(A)})_{\leq d}"] \\
(\mathbb D_{gr}(A)\otimes \mathbb D_{gr}(A)\otimes \mathbb D_{gr}(A)\otimes \mathbb D_{gr}(A))_{\leq d} \arrow[u,"(m_{\mathbb D_{gr}(A)}\otimes m_{\mathbb D_{gr}(A)})_{\leq d}"'] & (\mathbb D_{gr}(A)\otimes \mathbb D_{gr}(A))_{\leq d}\arrow[u,"(m_{\mathbb D_{gr}(A)})_{\leq d}"]\arrow[dl,swap,"(\Delta_{\mathbb D_{gr}(A)}\otimes \Delta_{\mathbb D_{gr}(A)})_{\leq d}"'] \\
(\mathbb D_{gr}(A)\otimes \mathbb D_{gr}(A)\otimes \mathbb D_{gr}(A)\otimes \mathbb D_{gr}(A))_{\leq d}\arrow[u,swap,"(\id\otimes c_{\mathbb D_{gr}(A),\mathbb D_{gr}(A)}\otimes\id)_{\leq d}"'] 
\end{tikzcd}
\]
showing that $\Delta_{\mathbb D_{gr}(A)}$ is an algebra morphism. Similarly, one shows that  $\epsilon_{\mathbb D_{gr}(A)}$ is an algebra morphism.
Hence,  $\mathbb D_{gr}(A)$ is an object in $\PB^{\leq d}(\Vu)$. One shows easily that morphisms in $\PB^{\leq d}(\Vu)$ are mapped to morphisms in $\PB^{\leq d}(\Vu)$ by $\mathbb D_{gr}$. Thus, the left duality functor in $\Vu$ induces a duality functor $\mathbb D_d$ on $\PB^{\leq d}(\Vu)$ for any $d\geq 0$.} 
\end{remark}

\subsection{The truncation functor}Let $d\in \NN$, let  $A=\bigoplus_{j\geq0}A_j$ be an object in $\PB(\Vu)$ with coproduct $\Delta_A$, counit $\epsilon_A$ and multiplication $m_A$, and let $A_{\leq d}\in \Vu_{gr}$. The restriction of $\epsilon_A$ and $\Delta_A$ to $A_{\leq d}$ equips it with a graded coalgebra structure. Since $A_{\leq d}$ identifies with $A/\bigoplus_{j>d}A_j$ as objects in $\Vu_{gr}$ and $A_{>d}:=\bigoplus_{j>d}A_j$ is a graded ideal, the restriction of $m_A$ to $A_{\leq d}\otimes A_{\leq d}$ equips $A_{\leq d}$ with a graded algebra structure. We denote by 
$\Theta_d(A)$ the algebra and coalgebra obtained this way, and by $m_{\Theta_d(A)}$, $\Delta_{\Theta_d(A)}$ and $\epsilon_{\Theta_d(A)}$ its multiplication, comultiplication and counit. Let $i,j\in\NN$ such that $i+j\leq d$. With the appropriate identifications, we have equality for the restrictions: $m_A|_{A_i\otimes A_j}=m_{\Theta_d(A)}|_{A_i\otimes A_j}$, $\epsilon_A|_{A_i}=\epsilon_{\Theta_d(A)}|_{A_i}$, $\Delta_A|_{A_i}=\Delta_{\Theta_d(A)}|_{A_i}$, and $c_{A,A}|_{A_i\otimes A_j}=c_{\Theta_d(A),\Theta_d(A)}|_{A_i\otimes A_j}$. Therefore, \eqref{eq:truncated} is verified for $\Theta_d(A)$. In addition, the truncation $(\ )_{\leq d}$ of any morphism $f\colon A\to A'$ in $\PB(\Vu)$ gives a morphism $\Theta_d(f)$ in  $\PB(\Vu)$, so $\Theta_d$ is a covariant functor $\PB(\Vu)\to \PB^{\leq d}(\Vu)$ which we call the $d$-th truncation functor. 

\medskip

\begin{example}
{\rm In the category of vector spaces with usual flip operator, if $V$ has basis $\{v_1,\,\ldots,\,v_n\}$, then $\Theta_2(T_*(V))$ is  generated as an algebra by the variables $v_i$, $v_{ij}$ for $i,j\in\{1,\,\ldots,\,n\}$ with relations $v_{ij}\circ v_l=v_l\circ v_{ij}=0$, $v_i\circ v_j=v_{ij}+v_{ji}$
for $i,j,l\in\{1,\,\ldots,\,n\}$. The coproduct is determined by $\Delta(v_i)=1\otimes v_i+v_i\otimes 1$ and $\Delta(v_{ij})=1\otimes v_{ij}+v_i\otimes v_j+v_{ij}\otimes 1$.}
\end{example}

\begin{remark}\label{rem:dual-truncation}{\rm Assume $\Vu$ is rigid. Then the truncation functor   $\Theta_d$ is compatible with the duality
functors $\mathbb{D}_{gr}\colon  \PB(\Vu)\to \PB(\Vu)$ and  $\mathbb{D}_d\colon  \PB^{\leq d}(\Vu)\to \PB^{\leq d}(\Vu)$, i.e., $\Theta_d\circ\mathbb{D}_{gr}=\mathbb{D}_d\circ\Theta_d$. Indeed, for $A=\bigoplus_{i\geq 0} A_i$ in $\PB(\Vu)$, at the level of objects in $\Vu_{gr}$ 
    \begin{align*}\Theta_d\circ\mathbb{D}_{gr}(A)=\Theta_d(\bigoplus_{i\geq 0} A^*_i)=\bigoplus_{i=0}^d A^*_i=\mathbb{D}_d
    (\bigoplus_{i= 0}^d A_i)=\mathbb{D}_d\circ\Theta_d(A).\end{align*} The comultiplication on objects in $\CB^{\leq d}(\Vu)$ lies only in degree $\leq d$, and $\Theta_d$ does not modify the coalgebra structure in degree $\leq d$, whence
    \begin{align*}\Delta_{\mathbb D_d\Theta_d(A)}=(\varphi_{A,A}^{gr})^{-1}\circ \mathbb D_{d}(m_{\Theta_d(A)})=((\varphi_{A,A}^{gr})^{-1}\circ  \mathbb D_{gr}(m_{\Theta_d(A)}))_{\leq d}= \Delta_{\Theta_d(\mathbb D_{gr}(A))}.\end{align*} 
    The multiplication in $\mathbb D_d\Theta_d(A)$ and in $\Theta_d(\mathbb D_{gr}(A))$ is trivial on $A^*_j\otimes A^*_l$ if $j+l>d$, and on terms in degree $\leq d$ we have
    \begin{align*}(m_{\mathbb D_d\Theta_d(A)})_{\leq d}=(\mathbb D_{d}(m_{\Theta_d(A)})\circ \varphi_{A,A}^{gr})_{\leq d}=
     (\mathbb D_{gr}(m_{A})\circ \varphi_{A,A}^{gr})_{\leq d}=m_{\Theta_d(\mathbb D_{gr}(A))}.
     \end{align*} 
    For unit and counit we have
    \begin{align*}
      u_{\mathbb D_d\Theta_d(A)}&=\left(\epsilon_{\Theta_d(A)}\right)^*=\left(\epsilon_{A}^*\right)_{\leq d}=u_{\Theta_d\mathbb D_{gr}(A)},\\
        \epsilon_{\mathbb D_d\Theta_d(A)}&=\left(u_{\Theta_d(A)}\right)^*=\left((u_{A})_{\leq d}\right)^*=
                \left(u_{A}^*\right)_{\leq d}=\epsilon_{\Theta_d\mathbb D_{gr}(A)},
    \end{align*}
    giving the equality at the level of objects in $\PB(\Vu)$. Similar arguments apply to morphisms. }
    \end{remark}

\subsection{The extension functor}

In order to define the extension functors, we need to generalize Radford's construction of the free bialgebra over a coalgebra from \cite[Section 5.3]{R} to the braided setting. Possibly passing to a completion, we assume that $\Vu$ is closed under countable direct sums. 

\medskip

Let $(C, \Delta_C, \epsilon_C)$ be a coalgebra in $\Vu$ and let $T(C)=\bigoplus_{j\geq0}C^{\otimes j}$ be equipped with the tensor algebra structure, so  $T(C)\otimes T(C)$ is an algebra in $\Vu$ as in Section \ref{S:BMC}. Let $\iota\colon C\to T(C)$  be the canonical inclusion.  By the universal property of $T(C)$, there are unique algebra morphisms $\epsilon_{T(C)}\colon T(C)\to \Bbbk$ and $\Delta_{T(C)}\colon T(C)\to T(C){\otimes}T(C)$ in $\Vu$ extending $\epsilon_C\circ\iota$ and $(\iota\otimes\iota)\circ \Delta_C$. This way $T(C)$ becomes a bialgebra in $\Vu$ enjoying the following universal property.
\begin{proposition}\label{prop:radford}
Let $C$ be a coalgebra in $\Vu$. For any bialgebra $A$ in $\Vu$ and any coalgebra morphism $f\colon C\to A$ there is a unique bialgebra morphism $\widetilde{f}\colon T(C)\to A$ satisfying $\widetilde{f}\circ\iota=f$. 
\end{proposition}
\begin{proof}This is proved exactly as \cite[Theorem 5.3.1 (b)]{R}.
\end{proof}
By construction $\iota$ is a one-to-one coalgebra morphism, so from now on we will identify $C$ and its image in $T(C)$ neglecting $\iota$. With this identification  $C\otimes C$ becomes a subcoalgebra of $T(C)\otimes T(C)$ because the restriction of the braiding $c_{T(C),T(C)}$ to $C\otimes C$ coincides with $c_{C,C}$. We observe that if $C=\bigoplus_{j\geq0}C_j$ is a coalgebra in $\Vu_{gr}$, then $T(C)$ is a bialgebra in $\Vu_{gr}$ where $T(C)_l=\bigoplus_{j\geq 0}(C^{\otimes j})_l$. 

\medskip

Let  now $d\in \mathbb N$ and let $A=\bigoplus_{l=0}^dA_l$ be an object in $\PB^{\leq d}(\Vu)$. Note that the tensor algebra $T(A)$ over $A$ regarded as a coalgebra is not connected, as in degree $0$ we have $\mathbf{1}_{\Vu}\oplus(\bigoplus_{j\geq1}A_0^{\otimes j})$. However, the elements $1_{\Bbbk}\in \mathbf{1}_{\Vu}=A^{\otimes 0}$ and $1_A^{\otimes j}\in A^{\otimes j}_0$ for $j\geq 1$ are all grouplike. 

\medskip

Denoting the multiplication in $A$ and $T(A)$, by $m_A$ and $m_{T(A)}$, respectively, we consider the algebra ideal $J_A=((m_{T(A)}-m_A)(A\otimes A)_{\leq d},\Bbbk(1_{\Bbbk}-1_A))$ in $T(A)$. 
\medskip
\begin{proposition}\label{prop:bialg}
With notation as above, $J_A$ is a coideal of $T(A)$.
\end{proposition}
\begin{proof}
First of all, $\Delta_{T(A)}(1_{\Bbbk}-1_A)\in T(A)_0\otimes J_A+J_A\otimes T(A)_0$ and $\epsilon_{T(A)}(1_{\Bbbk}-1_A)=0$ because the difference of two grouplikes $g$ and $h$ in a coalgebra is always $(g,h)$-primitive. Hence, $\Bbbk(1_{\Bbbk}-1_A)$ is a coideal. Observe  now that $(A\otimes A)_{\leq d}$ is a subcoalgebra of $A\otimes A$, whence of $T(A)\otimes T(A)$. By \eqref{eq:truncated} the morphism $(m_A)_{\leq d}$ 
gives a coalgebra morphism $(A\otimes A)_{\leq d}\to A_{\leq d}\subset A\subset T(A)$.  The restriction of $m_{T(A)}$ to $(A\otimes A)_{\leq d}$ is also a coalgebra morphism because $T(A)$ is a bialgebra. Since by \cite[Exercise 2.1.29]{R} the image of the difference of two coalgebra morphisms is a coideal, $(m_A-m_{T(A)})(A\otimes A)_{\leq d}$ is a coideal in $T(A)$. Hence,  $J_A=((m_A-m_{T(A)})(A\otimes A)_{\leq d},\Bbbk(1_{\Bbbk}-1_A))$ is a coideal in $T(A)$.
\end{proof}

\medskip

It follows from Proposition \ref{prop:bialg} that $\Psi_{d!}(A):=T(A)/J_A$ is a graded bialgebra, with grading inherited from $A$, and that $(T(A)/J_A)_0\simeq \Bbbk$. Hence $\Psi_{d!}(A)$ is an object in $\PB(\Vu)$. 

\medskip

Composing a morphism $f\colon A\to A'$ in $\PB^{\leq d}(\Vu)$ with the natural inclusion $A'\subset T(A')$ gives coalgebra morphism in $\Vu_{gr}$, which extends to a bialgebra morphism $T(f)$  by Proposition \ref{prop:radford} applied to $A$ and $T(A')$. By construction $T(f)$ is graded.

\begin{lemma}
    Let $A, A'$ be in $\PB^{\leq d}(\Vu)$ and let  $f\colon A\to A'$ be a morphism in $\PB^{\leq d}(\Vu)$. The bialgebra morphism $T(f)\colon T(A)\to T(A')$ maps the ideal $J_A$ into $J_{A'}$.
\end{lemma}
\begin{proof}
By construction, $T(f)(1_{\Bbbk}-1_A)=1_{\Bbbk}-f(1_A)=1_{\Bbbk}-1_{A'}\in J_{A'}$. In addition,
\begin{align*}
T(f)(m_A-m_{T(A)})(A\otimes A)_{\leq d}&=(f\circ m_A-T(f)\circ m_{T(A)}) (A\otimes A)_{\leq d}\\
&=(m_{A'}(f\otimes f)-m_{T(A')}(T(f)\otimes T(f))(A\otimes A)_{\leq d}\\
&=(m_{A'}-m_{T(A')})(f\otimes f)(A\otimes A)_{\leq d}\\
&\subset (m_{A'}-m_{T(A')})(A'\otimes A')_{\leq d}\subset J_{A'},
\end{align*}
where we have used that $f$ coincides with $T(f)$ on $m_A(A\otimes A)_{\leq d}$, that $f$ and $T(f)$ are algebra morphisms, that $(A\otimes A)_{\leq d}\subseteq A_{\leq d}\otimes A_{\leq d}$, and that $f\otimes f$ is graded.
\end{proof}

\medskip

For any morphism $f\colon A\to A'$ in $\PB^{\leq d}(\Vu)$ we then set  $\Psi_{d!}(f)\colon \Psi_{d!}(A)\to \Psi_{d!}(A')$ to be the  bialgebra morphism induced by $T(f)$. It is not hard to verify that $$\Psi_{d!}\colon  \PB^{\leq d}(\Vu)\to \PB(\Vu)$$ is a covariant functor, we call it the $d$-th extension functor. 

\medskip

\subsection{An adjunction}
 Let $d\in \NN$, let $A\in \PB^{\leq d}(\Vu)$ and let $\eta_A\colon A\to \Psi_{d!}(A)$ be the composition of the inclusion $A\subset T(A)$ with the canonical projection 
 $T(A)\to \Psi_{d!}(A)$. 

 \medskip 
 
 For $f \in\mathrm{Hom}_{\PB(\Vu)}(\Psi_{d!}(A),B)$, the composition $f\circ \eta_A\colon A\to B$ is a graded coalgebra morphism with image in $B_{\leq d}$. Now, 
 \begin{align*}
 f\circ \eta_A\circ m_A|_{(A\otimes A)_l}&=0=m_{\Theta_d(B)}((f\circ \eta_A)\otimes (f\circ \eta_A))|_{(A\otimes A)_l} &\mbox{ if $l>d$, and}\\
 f\circ \eta_A\circ m_A|_{(A\otimes A)_{\leq d}}&=f\circ m_{\Psi_{d!}(A)}\circ(\eta_A\otimes\eta_A)|_{(A\otimes A)_{\leq d}}\\
 &=m_B\circ 
 ((f\circ \eta_A)\otimes (f\circ \eta_A))|_{(A\otimes A)_{\leq d}}\\
 &=m_{\Theta_d(B)}\circ 
 ((f\circ \eta_A)\otimes (f\circ \eta_A))|_{(A\otimes A)_{\leq d}}.
 \end{align*}
 Hence, $(f\circ \eta_A)_{\leq d}$ lies in $\mathrm{Hom}_{\PB^{\leq d}(\Vu)}(A,\Theta_d(B))$.  We set 
  \begin{align}\label{eq:Xi}\Xi_{AB}\colon \mathrm{Hom}_{\PB(\Vu)}(\Psi_{d!} (A), B) &\longrightarrow\mathrm{Hom}_{\PB^{\leq d}(\Vu)}(A, \Theta_d (B))\\
 \nonumber f&\mapsto (f\circ \eta_A)_{\leq d}.\end{align}

 Let now $g\in\mathrm{Hom}_{\PB^{\leq d}(\Vu)}(A,\Theta_d (B))$. The composition of $g$ with the coalgebra inclusion $\tau_B\colon\Theta_d(B)\subset B$ extends to a bialgebra morphism $\widetilde{\tau_B\circ g}\colon T(A)\to B$ by Proposition \ref{prop:radford}. 

\begin{lemma}\label{lem:factor}With the above notation, the bialgebra morphism $\widetilde{\tau_B\circ g}\colon T(A)\to B$ factors through $\Psi_{d!}(A)$. 
\end{lemma}
\begin{proof}First of all $\tau_Bg(1_A)=1_{B}$ because $g$ is a morphism in $\PB^{\leq d}(\Vu)$, and $\widetilde{\tau_B\circ g}(1_{\Bbbk})=1_B$ because $\widetilde{\tau_B\circ g}$ is a bialgebra morphism, hence $\widetilde{\tau_B\circ g}(1_A-1_{\Bbbk})=0$. Moreover,
\begin{align*}
 \widetilde{\tau_B\circ g}&(m_A-m_{T(A)})(A\otimes A)_{\leq d}=(\tau_B\circ g\circ m_A - \widetilde{\tau_B\circ g}\circ m_{T(A)}) (A\otimes A)_{\leq d}\\
 &=m_B\circ \left(\left((\tau_B\circ g)\otimes (\tau_B\circ g)\right)-(\widetilde{\tau_B\circ g}\otimes \widetilde{\tau_B\circ g})\right)(A\otimes A)_{\leq d}=0
\end{align*}
where we have used that $\widetilde{\tau_B\circ g}$ is an algebra morphism extending $\tau_B\circ g$ on $A$.
\end{proof}
\medskip
Let $\overline{g}\colon \Psi_{d!}(A)\to B$ be the morphism induced by $\widetilde{\tau_B\circ g}\colon T(A)\to B$. We set
\begin{align}\label{eq:Gamma}\Gamma_{AB}\colon \mathrm{Hom}_{\PB^{\leq d}(\Vu)}(A, \Theta_d (B))&\longrightarrow \mathrm{Hom}_{\PB(\Vu)}(\Psi_{d!} (A), B)\\
 \nonumber g&\mapsto\overline{g}.\end{align}

\begin{proposition}\label{prop:adjunction}
 Let $d\geq0$. 
 The maps $\Xi_{AB}$ for $A\in\PB^{\leq d}(\Vu)$ and $B\in\PB(\Vu)$ combine to give a natural isomorphism
 \begin{equation}\label{eq:adjunction}
 \Xi\colon \mathrm{Hom}_{\PB(\Vu)}(\Psi_{d!}(\ ), \ ) \longrightarrow\mathrm{Hom}_{\PB^{\leq d}(\Vu)}(\ , \Theta_d (\ ))
\end{equation} whose inverse is $\Gamma=(\Gamma_{AB})$. Hence, $\Psi_{d!}$ is left adjoint to  $\Theta_d$.
\end{proposition}
\begin{proof}
Naturality in $A$ and $B$ follows from commutativity of the diagrams
\begin{equation*}\begin{tikzcd}
A' \arrow{r}{\eta_{A'}} \arrow[swap]{d}{\phi} & T(A')\arrow{r} \arrow{d}{T(\phi)} &\Psi_{d!}(A')\arrow{d}{\Psi_{d!}(\phi)} \\
A \arrow{r}{\eta_{A}}  & T(A)\arrow{r} &\Psi_{d!}(A')
\end{tikzcd}\hskip2cm
\begin{tikzcd}
B \arrow{r}{\psi} \arrow[swap]{d} & B' \arrow{d} \\
\Theta_d(B) \arrow{r}{\Theta_d(\psi)}& \Theta_d(B')
\end{tikzcd}
\end{equation*}
for any $\psi\in\mathrm{Hom}_{\PB(\Vu)}(B,B')$ and any $\phi\in \mathrm{Hom}_{\PB^{\leq d}(\Vu)}(A',A)$, where the vertical arrows in the diagram on the right are the natural projections. Let $f\in \mathrm{Hom}_{\PB(\Vu)}(\Psi_{d!}(A),B)$ and $g\in\mathrm{Hom}_{\PB^{\leq d}(\Vu)}(A,\Theta_d(B))$. Then, $\Gamma_{AB}\Xi_{AB}(f)=f$ and $\Xi_{AB}\Gamma_{AB}(g)=g$ by commutativity of the diagrams below.
\begin{equation*}
\begin{tikzcd}
A \arrow{r}{\Xi(f)} \arrow[swap]{d} & B_{\leq d}\arrow{r}{\tau_B} &B\\
T(A)\arrow{rr} \arrow{urr}[swap]{\widetilde{\tau_B\circ\Xi(f)}}&&\Psi_{d!}(A)\arrow{u}[swap]{f} 
\end{tikzcd}\hskip2cm
\begin{tikzcd}
A \arrow{r}{g} \arrow[swap]{d} & B_{\leq d}\arrow{r}{\tau_B} &B\\
T(A)\arrow{rr} \arrow{urr}[swap]{\widetilde{\tau_B\circ g}}&&\Psi_{d!}(A)\arrow{u}[swap]{\Gamma(g)} 
\end{tikzcd}
\end{equation*}
\end{proof}
\medskip

If $\Vu$ is  rigid, then $\PB(\Vu)$ and $\PB^{\leq d}(\Vu)$ have a duality, cf. Remark \ref{rmk:duality} and Subsection \ref{subsec:dual}. We denote by $\Psi_{d*}$ the functor from $\PB^{\leq d}(\Vu)$ to $\PB(\Vu)$ dual to $\Psi_{d!}$, that is, $\Psi_{d*}(A):=\mathbb D_{gr}(\Psi_{d!}(\mathbb D_d (A)))$  on objects and  $\Psi_{d*}(f):=\mathbb D_{gr}(\Psi_{d!}(\mathbb D_d(f)))$ on morphisms.
\begin{corollary}\label{cor:right-adjoint}
Assume that $\Vu$ is rigid. 
Then, for any $d\in \NN$, the functor $\Psi_{d*}$  is right adjoint  to  $\Theta_d$.
\end{corollary}
\begin{proof}  
The functors $\mathbb D_{gr}$ and $\mathbb D_d$ are contravariant and naturally isomorphic to their quasi-inverses. Therefore, by Proposition \ref{prop:adjunction} we have a chain of natural isomorphisms:
\begin{align*}
\Hom_{\PB(\Vu)}(B,\Psi_{d*}(A))&\simeq
\Hom_{\PB(\Vu)}(B, \mathbb D_{gr}(\Psi_{d!}(\mathbb D_d (A))))\simeq\Hom_{\PB(\Vu)}(\Psi_{d!}(\mathbb D_d(A)),\mathbb D_{gr}(B))\\
&\simeq\Hom_{\PB^{\leq d}(\Vu)}(\mathbb D_d(A),\Theta_d (\mathbb D_{gr}(B)))\simeq\Hom_{\PB^{\leq d}(\Vu)}(\mathbb D_d(\Theta_d (\mathbb D_{gr}(B))),A).
\end{align*}
The statement follows from Remark \ref{rem:dual-truncation}.\end{proof}

\subsection{The approximation functor} 
Let $d\in \NN$. In order to introduce the $d$-th approximation functor, we first analyze the components of the unit in the adjunction $\Psi_{d!} \dashv \Theta_d$. For $A$ in $\PB^{\leq d}(\Vu)$ this is $\Xi_{A,\Psi_{d!}}(A)(\id_{\Psi_{d!}(A)})=(\eta_A)_{\leq d}$.

\begin{proposition}\label{prop:estensione}
Let $d\geq0$. For any object $A$ in $\PB^{\leq d}(\Vu)$, the morphism $(\eta_A)_{\leq d}\colon A\to \Theta_d(\Psi_{d!}(A))$ is an isomorphism. Therefore  $\Theta_d$ is essentially surjective. 
\end{proposition}
\begin{proof}We first show that the kernel of the composition $\eta_A$ of the inclusion $A\to T(A)$ with the canonical projection $T(A)\to \Psi_{d!}(A)$ is trivial. The standard grading on $T(A)$, where $T^l(A)=A^{\otimes l}$, induces a filtration on $T(A)$ and on $\Psi_{d!}(A)$. 

For $m\geq 0$, let $\Psi_{d!}^{\leq m}(A)$ be the $m$-th term in the filtration and let $\mathrm{gr}(\Psi_{d!}(A))$ be the associated graded algebra, with $m$-th component $\mathrm{gr}(\Psi_{d!}(A))_m=\Psi_{d!}^{\leq m}(A)/\Psi_{d!}^{\leq m-1}(A)$ for $m\geq 1$ and $\mathrm{gr}(\Psi_{d!}(A))_0=\Bbbk 1_{\Bbbk}$.  Then, 
\begin{align*}
 \mathrm{gr}(\Psi_{d!}(A))_1=\left((\Bbbk\oplus A)/(1_{\Bbbk}-1_A)\right)/\Bbbk 1_{\Bbbk}=\Bbbk\oplus A_{>0}/\Bbbk.   
\end{align*}
For $m\geq0$, let $\phi_m\colon T^m(A)\to \mathrm{gr}(\Psi_{d!}(A))_m$ be the composition of the projection onto $\Psi^{\leq m}_{d!}(A)$ with the projection onto $\mathrm{gr}(\Psi_{d!}(A))_m$. Then $\phi:=\bigoplus_{m\geq0}\phi_m$ is a graded algebra morphism.  Identifying $A$ with  $T^1(A)$, we see that ${\rm ker}(\eta_A)\subseteq {\rm ker}(\phi_1)$ and, identifying $\mathrm{gr}(\Psi_{d!}(A))_1$ with $A_{>0}={\rm ker} (\epsilon_A)$, the map $\phi_1$ becomes $\id_A-\epsilon_A$, giving ${\rm ker}(\eta_A)\subseteq A_0=\Bbbk 1_A$. However, $\eta_A(\Bbbk 1_A)=\Bbbk 1_{\Psi_{d!}(A)}$, giving injectivity of $\eta_A$ whence of $(\eta_A)_{\leq d}$.

\medskip

Now, $\Psi_{d!}(A)$ is generated as an algebra by $\eta_A(A)$, so the same holds for its algebra quotient $\Theta_d(\Psi_{d!}(A))$. Therefore, the algebra morphism $(\eta_A)_{\leq d}\colon A\to \Theta_d(\Psi_{d!}(A))$ is surjective.  
\end{proof}

Since $(\eta_A)_{\leq d}$ is the identity in degree $1$, and since $\Bbbk 1_A$ is identified with $\mathbf{1}_\Vu=\Bbbk 1_{\Bbbk}$ we can and will identify $A$ with $\Theta_d(\Psi_{d!}(A))$ and $(\eta_A)_{\leq d}$ with the identity map.

 \medskip

We consider  the endofunctor  $$F_d:=\Psi_{d!}\circ \Theta_d\colon  \PB(\Vu)\to \PB(\Vu),$$
which we call the  {$d$-th approximation functor,} and the natural transformation $$\pi^d\colon F_d\to \id_{\CB(\Vu)}$$ which is the counit   of the adjuction $\Psi_{d!} \dashv \Theta_d$. Its component $\pi^d_B$ at $B$ in $\PB(\Vu)$ is $\Gamma_{\Theta_d(B),\Psi_{d!}(\Theta_d(B))}(\id_{\Theta_d(B)})$. 

 \begin{lemma}\label{rem:characterising}
We have the following properties:
\begin{enumerate}
\item[(1)] For any object $B$ in $\PB(\Vu)$, we have $\Theta_d (F_d (B))=\Theta_d (B)$.  
\item[(2)] For any  $B$ in $\PB(\Vu)$,  the map  $\pi^d_B$ is the unique morphism from $F_dB$ to $B$ such that  $\Theta_d(\pi_B^d)\colon \Theta_d(F_d(B))=\Theta_d (B)\to\Theta_d (B)$ is the identity map.    
\item[(3)] For any  $B,\,C$ in $\PB(\Vu)$ and any morphism $p\colon C\to B$ such that $ \Theta_d(p)\colon \Theta_d C\to \Theta_d B$ is an isomorphism,  there is a unique morphism $\widetilde{p}_B\colon F_d(B)\to C$ such that $p\circ \widetilde{p}_B=\pi^d_B$. 
\item[(4)] For any object $B$ in $\PB(\Vu)$ and any $e,d\in\NN$ such that $e\leq d$, there is a natural morphism $F_e(B)\to F_d(B)$ in $\PB(\Vu)$ whose restriction to $\Theta_e(B)$ is the identity. 
\end{enumerate}
\end{lemma}

\begin{proof}
(1) follows from Proposition \ref{prop:estensione} applied to $A=\Theta_d (B)$. 
By Proposition \ref{prop:estensione} we see that 
$\pi_B^d=\Gamma_{\Theta_d(B),\Psi_{d!}(\Theta_d(B))}(\id_{\Theta_d(B)})$ is the unique morphism  $\Psi_{d!}(\Theta_d(B))\to B$ in $\PB(\Vu)$ extending the identity on $B_{\leq d}$, i.e., such that $\Theta_d(\pi_B^d)=\id_{\Theta_d(B)}$, giving (2). 
We prove (3). Let $\widetilde{p}_B:=\Gamma_{\Theta_d(B),C}(\Theta_d(p)^{-1})$. By definition of $\Gamma$ it is the unique morphism in $\PB(\Vu)$ extending $\Theta_d(p)^{-1}$ to $F_d(B)$, so 
$p\circ \widetilde{p}_B$ satisfies (2), whence $p\circ \widetilde{p}_B=\pi^d_B$. By construction $\widetilde{p}_B$ is unique with this property. In order to prove (4), we observe that there is a composition of coalgebra inclusions  $\Theta_e(B)\subseteq\Theta_d(B)\subset T(\Theta_d(B))$ giving rise to a bialgebra inclusion $T(\Theta_e(B))\to T(\Theta_d(B))$. Since the restrictions of $m_{\Theta_d(B)}$ and $m_{\Theta_e(B)}$ to $(\Theta_e(B)\otimes\Theta_e(B))_{\leq e}$ coincide, 
we have an inclusion of the ideals 
\begin{align*}
 (((m_{T(\Theta_e(B))}-m_{\Theta_e(B)})(\Theta_e(B)\otimes\Theta_e(B))_{\leq e}) \subseteq  (((m_{T(\Theta_d(B))}-m_{\Theta_d(B)})(\Theta_d(B)\otimes\Theta_d(B))_{\leq d}).
\end{align*} Thus, the composition $T(\Theta_e(B))\to T(\Theta_d(B))\to \Psi_{d!}\Theta_d(B)=F_d(B)$ factors through $F_e(B)$. Using Proposition \ref{prop:estensione} we identify $\Theta_e(F_e(B))=\Theta_e(B)$ and $\Theta_e(F_d(B))=\Theta_e(B)$, and through these identifications, the restriction of the induced morphism $F_e(B)\to F_d(B)$ up to degree $e$ is the identity. 

\end{proof}

\begin{example}\label{ex:approx}
  {\rm  
  \begin{enumerate}
      \item If $V$ is an object in $\Vu$ concentrated in degree $1$ then the $d$-th approximation of $T_!(V)$  is $T_!(V)$ for any $d\geq 1$. 
      \item  In the category of vector spaces with usual flip operator, if $V$ has basis $\{v_1,\,\ldots,\,v_n\}$, then $F_2(T_*(V))$ is  generated as an algebra by the variables $x_i$, $x_{ij}$ for $i,j\in\{1,\,\ldots,\,n\}$ with relations $x_ix_j-x_{ij}-x_{ji}=0$ for $i,j\in\{1,\,\ldots,\,n\}$. The coproduct is determined by $\Delta(x_i)=1\otimes x_i+x_i\otimes 1$ and $\Delta(x_{ij})=1\otimes x_{ij}+x_i\otimes x_j+x_{ij}\otimes 1$.
      \item The $0$-th approximation  of any $B\in\PB(\Vu)$ is $\mathbf{1}_\Vu$.
      \item The first approximation  of $B=\bigoplus_{j\geq 0} B_j\in\PB(\Vu)$ is  $F_1(B)=T_!(B_1)$.
        \end{enumerate}}
\end{example}

The functor $F_d$ generalizes the construction of the $d$-th approximation from bialgebras generated in degree $1$  to all connected bialgebras:

\begin{theorem}\label{Thm1}Let $d\geq 2$. Then the restriction of the functor $F_d$ to $\CB^1(\Vu)$ is naturally isomorphic to the $d$-atic covering functor \eqref{eq:approximation}. 
\end{theorem}
\begin{proof}  
Let $B=T_!(B_1)/J\in \CB^1(\Vu)$ and let $\widehat{B}_d=T_!(B_1)/(\bigoplus_{j\leq d}T_!^j(B_1)\cap J)$ be its $d$-th cover. By Lemma \ref{rem:characterising} (4) and Example \ref{ex:approx} (4) there is a natural morphism $\widehat{\zeta_B}\colon T_!(B_1)\to F_d(B)$ in $\PB(\Vu)$ whose restriction to $\Theta_1(B)$ is the identity. Now, as $B$ is generated by $B_1$,  so are the algebra quotient $\Theta_d(B)$ of $B$, and the extension $F_d(B)$ of $\Theta_d(B)$. Therefore $\widehat{\zeta_B}$ is surjective.

\medskip

We now show that $\widehat{\zeta_B}$ factors though $\widehat{B}_d$.
By definition of $J_{\Theta_d(B)}$ and Proposition \ref{prop:estensione}, and with the appropriate identifications, the following equalities hold for $j\leq d$ and $b_1,\ldots,b_j\in B_1$ \begin{equation*}\widehat{\zeta_B}(b_1\otimes\cdots\otimes b_j)=m_{\Theta_d(B)}^{(j-1)}(b_1\otimes\cdots\otimes b_j)=m_B^{(j-1)}(b_1\otimes\cdots\otimes b_j).\end{equation*}
Hence, for any $j\leq d$ and any $x\in T_!^j(B_1)\cap J$ there holds 
$\widehat{\zeta_B}(x)=m^{(j-1)}_Bx=0$, whence $\widehat{\zeta_B}(\bigoplus_{j\leq d}T_!^j(B_1)\cap J)=0$. Thus, $\widehat{\zeta_B}$ induces a surjective morphism $\overline{\zeta}_B\colon \widehat{B}_d\to F_d(B)$ in $\PB(\Vu)$. By construction, it restricts to the identity on $B_1$.  

\medskip

We now find an inverse for $\overline{\zeta}_B$. Observe that
\begin{align*}
\Theta_d(\widehat{B}_d)&=(T_!(B_1)/(\bigoplus_{j\leq d}T^j_!(B_1)\cap J))/(\bigoplus_{l>d}\widehat{B}_d)_l=T_!(B_1)/(\bigoplus_{j>d}T^j_!(B_1)\oplus(\bigoplus_{j\leq d}T^j_!(B_1)\cap J)),\\
\Theta_d(B)&=(T_!(B_1)/J)/\bigoplus_{l>d}B_l=T_!(B_1)/(\bigoplus_{j>d}T^j_!(B_1)+J),\end{align*} respectively. Since $(\bigoplus_{j\leq d}T^j_!(B_1)\cap J))+T_!^{\geq d}(B_1)=J+T_!^{\geq d}(B_1)$, we have $\Theta_d(\widehat{B}_d)=\Theta_d(B)$.

Let $\Gamma_{\Theta_d(B),\widehat{B}_d}$ be as in \eqref{eq:Gamma}. Invoking Proposition \ref{prop:adjunction}, we define 
\begin{align*}\xi_B:=\Gamma_{\Theta_d(B),\widehat{B}_d}(\id_{\Theta_d(B)})\in \mathrm{Hom}_{\PB(\Vu)}(\Psi_{d!}(\Theta_d(B)), \widehat{B}_d ) =\mathrm{Hom}_{\PB(\Vu)}(F_d(B), \widehat{B}_d ).\end{align*}
By construction, the restriction of $\xi_B$ to $\Theta_d(B)$, whence also to $B_1$, is the identity, hence the restriction of $\xi_B\circ \overline{\zeta}_B$ is the identity. Since $\widehat{B}_d$ is generated by $B_1$, we have $\xi_B\circ \overline{\zeta}_B=\id_{\widehat{B}_d}$. Therefore, $ \overline{\zeta}_B$ is also injective, and its left inverse $\xi_B$ is also a right inverse. 

\medskip

By naturality of the construction of  $\overline{\zeta}_B$, we see that $\overline{\zeta}=(\overline{\zeta}_B)$ is a natural isomorphism of functors. 
\end{proof}

Note that $A$ is finitely presented if and only if $A=F_d(A)$ for some $d$ with $\dim A_{\leq d}<\infty$. Indeed, the condition $A=F_d(A)$ implies that the relations are generated up to degree  $d$, while the condition $\dim A_{\leq d}<\infty$ implies that one can choose finitely many generators.

\section{Some equivalences of categories}

\subsection{Inverse limits of categories} 
We call a {projective system of categories  indexed by} $\NN$ a pair $(({\mathcal C}_d)_{d\in {\mathbb N}},(\pi_{de})_{d\leq e})$
where $({\mathcal C}_d)_{d\in {\mathbb N}}$ is a collection of categories indexed by ${\mathbb N}$, and   $\pi_{de}\colon{\mathcal C}_e\to {\mathcal C}_d$ for $d\leq e$ are functors satisfying $\pi_{de}\circ\pi_{ef}=\pi_{df}$ if $d\leq e\leq f$ and $\pi_{dd}=\id$ for all $d\in \NN$.

In the spirit of \cite[\S 5.1]{Schn}, we define the {inverse limit category} $\varprojlim_{d \in \mathbb{N}} \mathcal{C}_d$ \label{limit}, as the category whose objects are pairs $\left((E_d)_{d \in \mathbb{N}}, (\phi_{de})_{d \leq e}\right)$, where each $E_d$ is an object in $\mathcal{C}_d$, and each $\phi_{de}$ is an isomorphism $\phi_{de} \colon \pi_{de}(E_e) \to E_d$ in $\mathcal{C}_d$, satisfying the compatibility condition $\phi_{de} \circ \pi_{de}(\phi_{ef}) = \phi_{df}$ for $d \leq e \leq f$. In particular, $\phi_{ee}\circ\phi_{ee}=\phi_{ee}$. Hence, in the categories we will be interested in,  the maps $\phi_{ee}$ will be the identity.  Morphisms from $\left((E_d)_{d\in \NN},(\phi_{de})_{d\leq e}\right)$ to $\left((E'_d)_{d\in \NN},(\phi'_{de})_{d\leq e}\right)$ are defined as families of morphisms $\psi_d\colon E_d\to E_d'$ in ${\mathcal C}_d$ for each $d\in {\mathbb N}$ such that the diagram below commutes for each $d\leq e$:
\begin{equation}\label{eq:CDlimit}
\begin{CD}
\pi_{de}(E_e)@>{\phi_{de}}>>E_d\\
@V{\pi_{de}(\psi_e)}VV @VV{\psi_d}V\\
\pi_{de}(E'_e)@>{\phi'_{de}}>>E'_d\\
\end{CD}
\end{equation}
where the vertical arrow on the left is obtained by applying the functor $\pi_{de}$ to the morphism $\psi_e\colon E_e\to E_e'$.
This construction is a special case of the {quasi-limit}, \cite[p. 217]{G}.

\bigskip

The truncation $\bigoplus_{j\leq d}A_j\mapsto\bigoplus_{j\leq e}A_j$ induces a family of functors $(\ )_{\leq e,d}\colon \PB^{\leq d}(\Vu)\to \PB^{\leq e}(\Vu)$ for $e\leq d$, and $((\PB^{\leq d}(\Vu))_{d\in\mathbb N}, (\ )_{\leq e,d})_{e\leq d})$ is a projective system of categories.  Observe that $(\ )_{\leq e,d}\Theta_d=\Theta_e$ for any $e\leq d$. 

\begin{proposition}\label{prop:limit-PB}
The inverse limit  category $\varprojlim_{d\in {\mathbb N}} \PB^{\leq d}(\Vu)$ is equivalent to $\PB(\Vu)$. 
\end{proposition}
\begin{proof}
We define a functor $\Theta\colon \PB(\Vu)\to \varprojlim_{d\in {\mathbb N}} \PB^{\leq d}(\Vu)$ as follows. For an object $A$ in $\PB(\Vu)$ we set
$\Theta(A):=\left((\Theta_d(A))_{d\geq0},\, \id_{l\leq d}\right)$ where $\id_{l\leq d}$ denotes the identity morphism in $\PB^{\leq l}(\Vu)$. For a morphism $f\colon A\to B$ in $\PB(\Vu)$, we set ${\Theta}(f):=(\Theta_d(f)_{d\in\NN})\colon\Theta(A)\to \Theta(B)$. By construction the functor $\Theta$ is fully faithful. We show that $\Theta$ is essentially surjective. 

Let $\left((E^d)_{d\geq0}, (\phi_{ld})_{l\leq d}\right)$ be an object in $\varprojlim_{d\in {\mathbb N}} \PB_{\leq d}(\Vu)$, with $E^d=\bigoplus_{i=0}^dE^d_i$. We define the object  $\widetilde{E}$ of $\Vu_{gr}$ as follows: $\widetilde{E}=\bigoplus_{d\geq 0}E^d_d$. Then $\widetilde{E}_0=E_0^0={\mathbf 1}_\Vu$. The morphisms 
$$m_{i,j}:=m_{E^{i+j}}\circ(\phi_{i,i+j}^{-1}\otimes\phi^{-1}_{j,i+j})\colon\widetilde{E}_i\otimes \widetilde{E}_j=E_i^i\otimes E_j^j\to E_{i+j}^{i+j}=\widetilde{E}_{i+j}$$
combine to a graded, associative multiplication morphism on $\widetilde{E}$.  Similarly, the maps 
$$\widetilde{\Delta}_{d}:=\bigoplus_{i+j=d}(\phi_{i,d}\otimes\phi_{j,d})\Delta_d\colon \widetilde{E}_d=E_d^d\to \bigoplus_{i+j=d}E_i^i\otimes E_j^j=\bigoplus_{i+j=d}\widetilde{E}_i\otimes \widetilde{E}_j$$
combine to a graded, coassociative comultiplication morphism on $\widetilde{E}$. The counit is given by taking the counit on each $E_d^d$. A tedious computation shows that $\widetilde{E}$ is an object in $\PB(\Vu)$ and that the sequence of bialgebra morphisms $f^d\colon \widetilde{E}_{\leq d}\to E^d$ given by $f^d:=\bigoplus_{i\leq d}f^d_l=\bigoplus_{l\leq d}\phi_{ld}^{-1}|_{E^d_l}$ combine to give an isomorphism $\Theta(\widetilde{E})\to \left((E^d)_{d\geq0}, (\phi_{ld})_{l\leq d}\right)$.
\end{proof}

\subsection{An alternative category of truncated bialgebras}

We describe here an alternative approach to the truncated versions of $\PB(\Vu)$. Let $d$ be a non-negative integer.  
The $d$-{th truncated category} ${\mathcal V}_{\leq d}$ of ${\mathcal V}_{gr}$ is the category with the same objects as  ${\mathcal V}_{gr}$ and morphisms $f=(f_j)_{0\leq j\leq d}$, i.e., only in degrees  less than or equal to $d$. 

In other words, we disregard graded components corresponding to terms greater than $d$. In the language of \cite[2.2]{C}, the category ${\mathcal V}_{\leq d}$ is the braided monoidal category obtained from $\Vu_{gr}$ by taking the quotient with respect to the tensor ideal consisting of graded morphisms that are $0$ up to degree $d$.  The unit in ${\mathcal V}_{\leq d}$ is the same as in  ${\mathcal V}_{gr}$. The 
braiding in ${\Vu}_{\leq d}$ is given by 
$\gamma_{A,B}=(\bigoplus_{i+j=k}c_{A_i,B_j})_{0\leq k\leq d}$,  where $c_{A_i,B_j}$ is the braiding in $\Vu$. Note that ${\mathcal V}_{\leq 0}=\Vu$ with the usual unit and braiding.

\medskip

If $\Vu$ is rigid, then each category $\Vu_{\leq d}$ has a duality induced from $\mathbb D_{gr}$ and the natural isomorphism $\varphi$ from Section \ref{subsec:dual} induces a braided monoidal structure on the duality functor. Hence, each category  $\PB(\Vu_{\leq d})$ of connected bialgebras in $\Vu_{\leq d}$ has  a duality, \cite[Remark 3.2.13]{HS}.

\begin{proposition}\label{prop:equivalence-bialgebras}
Let $d\in \mathbb N$. The identity functor on $\Vu_{gr}$ induces 
an equivalence of categories $\mathcal R_d\colon  \PB^{\leq d}(\Vu)\to \PB(\Vu_{\leq d})$. If $\Vu$ is rigid, $\mathcal R_d$ is compatible with duality. 
\end{proposition}
\begin{proof} 
An object $A$ in $\PB^{\leq d}(\Vu)$ is an object in $\Vu_{gr}$, hence an object in $\Vu_{\leq d}$, and $A_0={\mathbf 1}_\Vu={\mathbf 1}_{\Vu_{\leq d}}$. The multiplication $m_A$, the comultiplication $\Delta_A$ and the counit $\epsilon_A$ of $A$ are non-trivial only in degree $\leq d$, hence they can be identified with morphisms in $\Vu_{\leq d}$ with which $A$ becomes a graded algebra and coalgebra in $\Vu_{\leq d}$. Finally, \eqref{eq:truncated} implies that $A$ is a connected bialgebra in $\Vu_{\leq d}$. Morphisms in $\PB^{\leq d}(\Vu)$ are graded algebra and coalgebra morphisms in $\Vu_{\leq d}$, hence the identity functor induces a fully faithful functor $\PB^{\leq d}(\Vu)\to \PB(\Vu_{\leq d})$. Note that any algebra $A$
in ${\mathcal V}_{\leq d}$ is isomorphic to its truncation $A_{\leq d}$, hence the functor is essentially surjective. The statement about duality follows because duality in $\Vu_{gr}$ and $\Vu_{\leq d}$ coincides on objects and morphisms up to degree $d$. 
\end{proof}

For $l\leq d$, let  $\pi_{ld} \colon \Vu_{\leq d}\to \Vu_{\leq l}$ be the functors given by the identity on objects and on morphisms by disregarding the components greater than $l$. Then $((\Vu_{\leq d})_{d\in\NN}, (\pi_{ld})_{l\leq d})$ is a projective system of categories whose limit  $\varprojlim_{d\in {\mathbb N}} \Vu_{\leq d}$ is equivalent to $\Vu_{gr}$. By construction, for the restriction of $\pi_{ld}$ to $\PB(\Vu_{\leq d})$ there holds $\mathcal R_l(\ )_{\leq l,d}=\pi_{ld}\mathcal R_d$. Then, by Proposition \ref{prop:limit-PB} the inverse limit of the projective system $((\PB^{\leq d}(\Vu))_{d\geq 1},(\pi^{ld})_{l\leq d})$ is $\PB(\Vu)$.

\section{Compatibility with equivalences}
In this section, we  show that the $d$-th approximation functors behave well with respect to equivalence of braided categories and we apply it to the equivalence between categories of Yetter-Drinfeld modules of Hopf algebras related by a cocycle-twist,  \cite{MO,AFGV}.

\medskip

\begin{theorem}
    \label{thm:equivaprox}
    Let $(\mathcal{T},\omega)$ be an equivalence of additive braided categories from $\Vu$ to $\Vu'$ with quasi-inverse $(\mathcal{S},\alpha)$. Let $\Upsilon$  be a natural isomorphism of braided monoidal functors from $(\mathcal{S}\circ \mathcal{T}, \mathcal{S}(\omega_{-,-})\circ \alpha_{\mathcal{T}(-),\mathcal{T}(-)})$ to the identity functor $(\id_{\Vu}, \id)$.  For $d\in \NN$, let $\mathcal{T}\colon\CB(\Vu)\to \CB(\Vu')$,  and $\mathcal{T}^d\colon\CB^{\leq d}(\Vu)\to \CB^{\leq d}(\Vu')$, $\mathcal{S}\colon\CB(\Vu') \to \CB(\Vu)$ and $\mathcal{S}^d\colon\CB^{\leq d}(\Vu') \to \CB^{\leq d}(\Vu)$ be the induced functors. Let $\Theta_d \colon  \PB(\Vu)\to \PB^{\leq d}(\Vu)$,\,$\Theta'_d \colon  \PB(\Vu')\to \PB^{\leq d}(\Vu')$,
$\Psi_{d!}\colon \PB^{\leq d}(\Vu)\to \PB(\Vu)$, $\Psi'_{d!}\colon \PB^{\leq d}(\Vu')\to \PB(\Vu')$,
$F_d\colon  \PB(\Vu)\to \PB(\Vu)$ and $F'_d\colon  \PB(\Vu')\to \PB(\Vu')$ denote the $d$-th truncation, $d$-th extension and $d$-th approximation functors in $\Vu$ and $\Vu'$,  respectively. Then
    \begin{enumerate}
        \item $\bigoplus_{i=0}^d\Upsilon$ is a natural isomorphism from $\SSS^d\circ \Theta'_d\circ \mathcal{T}$ to $\Theta_d$;
        \item $\Psi_{d!}$ is naturally isomorphic to  $\SSS\circ \Psi'_{d!}  \circ \mathcal{T}^d$;
        \item $F_d$ is naturally isomorphic to $\SSS\circ  F'_d    \circ \mathcal{T}$.
        \item For $A$ in $\PB(\Vu)$, we have $A\simeq F_d(A)$ if and only if $\mathcal T(A)\simeq F_d'\mathcal T(A)$.
 \item  If $\mathcal T\circ\mathcal S=\id_{\Vu'}$ and $\mathcal S\circ\mathcal T=\id_{\Vu}$, that is, if $\Vu$ and $\Vu'$ are isomorphic, then we have uniquely determined identifications
 \begin{equation}\Theta_d=\SSS^d\circ \Theta'_d\circ \mathcal{T}, \quad \Psi_{d!}=\SSS\circ \Psi'_{d!}  \circ \mathcal{T}^d, \quad F_d=\SSS\circ  F'_d    \circ \mathcal{T}.\end{equation}
    \end{enumerate}
 \end{theorem}

\begin{proof}
  (1)  Let $A=\bigoplus_{i\geq 0}A_i$ be an object in $\CB(\Vu)$. Then $\mathcal{T}(A)=\bigoplus_{i\geq 0}\mathcal{T}(A_i)$. Hence $\Theta'_d(\mathcal{T}(A))=\bigoplus_{i=0}^d\mathcal{T}(A_i)$ as coalgebras. Therefore $\SSS^d(\Theta'_d(\mathcal{T}(A)))=\bigoplus_{i=0}^d \SSS(\mathcal{T}(A_i))\simeq \bigoplus_{i=0}^d A_i$, where the last isomorphism is  $\bigoplus_{i=0}^d\Upsilon_{A_i}$.
Now, $\mathcal{T}(\bigoplus_{j>d}A_j)=\bigoplus_{j>d}\mathcal{T}(A)_j$, so $\Theta'_d(\mathcal{T}(A))=\mathcal{T}(A)/\bigoplus_{j>d}\mathcal{T}(A_j)$ as algebras. Hence $\bigoplus_{i=0}^d\Upsilon_{A_i}$ induces an algebra isomorphism $\SSS^d(\Theta'_d(\mathcal{T}(A)))=\SSS(\mathcal{T}(A))/(\bigoplus_{j>d}\SSS(\mathcal{T}(A)))\simeq \Theta_d(A)$. Since  $\Upsilon$ is  a natural isomorphism, then $\bigoplus_{i=0}^d\Upsilon$ is a natural isomorphism from $\SSS^d\circ \Theta'_d\circ \mathcal{T}$  to $\Theta_d$.\\
(2)
For $A\in \PB^{\leq d}(\Vu)$ and $B\in \PB(\Vu)$ we have the sequence of natural isomorphisms
   \begin{equation}\label{eq:adj}
   \begin{split}&\mathrm{Hom}_{\CB(\Vu)}(\SSS\circ \Psi'_{d!}\circ \mathcal{T}^dA,B)\simeq \mathrm{Hom}_{\CB(\Vu)}(\SSS\circ \Psi'_{d!}\circ\mathcal{T}^dA, \mathcal{S}\circ \mathcal{T}B)\\
   &\simeq \mathrm{Hom}_{\CB(\Vu')}(\Psi'_{d!}\mathcal{T}^dA,\mathcal{T}B) \simeq \mathrm{Hom}_{\CB^{\leq d}(\Vu')}(\mathcal{T}^dA, \Theta'_d\mathcal{T}B)\\
   &\simeq \mathrm{Hom}_{\CB^{\leq d}(\Vu')}(\mathcal{T}^dA, \mathcal{T}^d\circ \SSS^d\circ \Theta'_d\mathcal{T}B)\simeq \mathrm{Hom}_{\CB^{\leq d}(\Vu')}(\mathcal{T}^dA, \mathcal{T}^d\Theta_dB)\\
   &\simeq \mathrm{Hom}_{\CB^{\leq d}(\Vu')}(A, \Theta_dB)
    \end{split}
    \end{equation}
    where in the second isomorphism we use that $\SSS$ is fully faithful, in the third the adjunction $\Psi'_{d!} \dashv \Theta'_d$, in the fourth that $\mathcal{T}^d$ is the quasi inverse of $\mathcal{S}^d$, in the fifth we compose with $\mathcal{T}^d(\bigoplus_{i=0}^d \Upsilon_{B_i})$ and in the last we use that $\mathcal{T}^d$ is fully faithful.
    Hence  $\SSS\circ \Psi'_{d!}\circ \mathcal{T}^d$ is left adjoint to $\Theta_d$, and so it is naturally isomorphic to $\Psi_{d!}$ by \cite[Corollary 1]{MCL}.\\
    (3) Follows form (1) and (2) and (4) follows from (3).
    Finally, observe that if $\mathcal T$ and $\mathcal S$ are mutual inverses, we may choose $\Upsilon$ to be the identity, so $\Theta_d=\SSS^d\circ \Theta'_d\circ \mathcal{T}$. In this case, for $A\in\PB^{\leq d}(\Vu)$, the natural isomorphism $\Psi_{d!}'\mathcal T^ d(A)\to \mathcal{T}\Psi_{d!}(A)$ is the unique morphism corresponding to $\id_A$ through the chain of identifications
\begin{align*}\mathrm{Hom}_{\PB^{\leq d}(\Vu)}(A,A)&\equiv \mathrm{Hom}_{\PB^{\leq d}(\Vu')}(\mathcal T^d(A),\mathcal T^d \Theta_d\Psi_{d!}(A))\\
&\equiv \mathrm{Hom}_{\PB^{\leq d}(\Vu')}(\mathcal T^d(A), \Theta'_d\mathcal T\Psi_{d!}(A))\simeq \mathrm{Hom}_{\PB(\Vu')}(\Psi_{d!}'\mathcal T^d(A),\mathcal T\Psi_{d!}(A)),
\end{align*}
where the last isomorphism is given by \eqref{eq:Gamma}. In other words,  it is the unique morphism extending the identity, up to identifications.

The last equality immediately follows.    \end{proof}

\subsection{Cocycle twisting} 
 Let $H$ be a Hopf algebra with multiplication $m_H$, comultiplication $\Delta$ and counit $\epsilon$ and invertible antipode. We adopt a Sweedler's like notation for coproducts and coactions, omitting the summation symbol. We recall that a convolution invertible linear map $\sigma: H\otimes H\to \Bbbk$ is 
 called a {normalized} $2$-{cocycle} for $H$ if 
  \begin{equation}\label{eq:sigma1}
      \sigma(x_{1}\otimes y_{1})\sigma(m_H(x_{2}\otimes y_{2})\otimes z)=\sigma(y_{1}\otimes z_{1})\sigma (x\otimes m_H(y_{2}\otimes z_{2})),\quad \forall x,y,z\in H,\quad \sigma(1\otimes 1)=1.
     \end{equation}
If $H=\Bbbk G$ is the group algebra of a finite group $G$, then a $2$-cocycle is the linear extension of a usual  $2$-cocycle on $G$: \begin{equation}\label{eq:cocycle}\sigma(x,y)\sigma(xy,z)=\sigma(y,z)\sigma(x,yz),\quad\sigma(x,1_G)=\sigma(1_G,x)=1,\quad\forall x,\,y,\,z\in G.\end{equation}
The set of normalized $2$-cocycles on $H$ with values in $\Bbbk$ is denoted by $Z^2(H,\Bbbk)$.

\medskip

We recall from \cite[Theorem 1.6]{Doi} that the  twist of $H$ by  a $2$-cocycle $\sigma$ is the Hopf algebra $H_\sigma$  with same underlying vector space and coproduct as $H$, and multiplication given by \begin{equation*}
 m_{H_\sigma}(x\otimes y):= \sigma(x_{1}\otimes y_{1})m_H(x_{2}\otimes y_{2})\sigma^{-1}(x_{3}\otimes y_{3}),\quad x,y\in H, \end{equation*}
where $\sigma^{-1}$ denotes the convolution inverse of $\sigma$. Hence, if $H$ is cocommutative, $H\simeq H_\sigma$ as Hopf algebras.

\medskip

According to \cite[Theorem 2.7]{MO} there is a braided monoidal equivalence $(\mathcal{T}_{\sigma}, \omega)$  between the categories of Yetter-Drinfeld modules $\ydH$ to $\ydHs$.  For an object $V$ in $\ydH$ with $H$-action $\rhd$ and coaction $\lambda\colon V\to H\otimes V$ mapping $v\in V$ to $v_{-1}\otimes v_0$, the object $\mathcal{T}_{\sigma}(V)$ has underlying vector space $V$, coaction $\lambda$, and $H_\sigma$-action given by 
\begin{align*}h\rhd_{\sigma}v:=\sigma(h_1\otimes v_{-1})(h_2\rhd v_0)_0\sigma^{-1}((h_2\rhd v_0)_{-1}\otimes h_3),  \mbox{ for }h\in H, \mbox{ and }v\in V.\end{align*} At the level of morphisms $\mathcal{T}_{\sigma}$ is the identity. The components of the natural transformation $\omega$ are given by
\begin{align*}
    \omega_{V,W}\colon V_{\sigma}\otimes W_{\sigma}&\to (V\otimes W)_{\sigma}\\
    (v\otimes w)&\mapsto \sigma(v_{-1}\otimes w_{-1})v_0\otimes w_0,
\end{align*}
for all $v,w\in \ydH$.

 Then,  $(\mathcal{T}_{\sigma},\omega)$ induces a  functor $\mathcal{T}_{\sigma}\colon \CB(\ydH)\to \CB(\ydHs)$. According to Remark \ref{rmk:bmbmf}, for $A\in \PB(\ydH)$ with multiplication $m_A$, the multiplication in $\mathcal T_\sigma(A)$ is given by $a\otimes a'\mapsto\sigma(a_{-1}\otimes a'_{-1})m_A(a_0\otimes a'_0)$, i.e., it is the algebra twisted by $\sigma$.

\begin{theorem}\label{thm:twistapprox}
Let $d\in \NN$. Let $H$ be a Hopf algebra with invertible antipode, $\sigma\in Z^2(H,\Bbbk)$, and $F_d, F_d'$ be the $d$-th approximation functors on $\PB(\ydH)$ and $\PB(\ydHs)$, respectively. Then 
\begin{enumerate}
\item $F'_d\circ \mathcal{T}_{\sigma}=\mathcal{T}_{\sigma}\circ F_d$.
\item For $A\in \PB(\ydH)$ there holds
$A\simeq F_d(A)$ if and only if $F'_d(\mathcal{T}_{\sigma}(A))\simeq \mathcal{T}_{\sigma}(A)$. 
\end{enumerate}
\end{theorem}
\begin{proof} If $\sigma$ lies in $Z^2(H,\Bbbk)$, then  $\sigma^{-1}$ lies in $Z^2({H_\sigma},\Bbbk)$,  \cite[Remark 2.8.3]{HS} and $(H_\sigma)_{\sigma^{-1}}=H$.
Hence $\mathcal{T}_{\sigma}$ and $\mathcal{T}_{\sigma^{-1}}$ are mutual inverses, see also \cite[\S 3.4]{AFGV}. We conclude applying  Theorem \ref{thm:equivaprox} (4) and (5).
\end{proof}

\subsection{A concrete example: Nichols algebras and cocycle twist on racks}
We discuss here the very concrete case of Nichols algebras corresponding to twist equivalent rack cocycles. We introduce the necessary terminology. 

\begin{definition} A {rack} is a pair $(X,\trid)$ where  $X$ is a non-empty  set and  $\trid$ is a binary operation on $X$ satisfying  the following conditions:
\begin{enumerate}
\item The map $x\trid -\colon X\to X$ is bijective for any $x\in X$;
\item $x\trid(y\trid z)=(x\trid y)\trid (x\trid z)$ for any $x,\,y,\,z\in X$. 
\end{enumerate}
\end{definition}
 The prototypical example of a rack is a union of conjugacy classes in a finite group together with the conjugation operation. 
\begin{definition}\label{def:cocycledegn}
A $2$-{cocycle of degree} $n$ on a rack $(X,\trid)$ is a 
map $q\colon X\times X\to \GL_n(\Bbbk)$  satisfying
\begin{equation}\label{eq:codeg}
    q({x,y\trid z})q({y,z})=q({x\trid y, x\trid z})q({x,z})
\end{equation}
for $x,\,y,\,z\in X$. 
\end{definition}

For a rack $(X,\trid)$   and   a $2$-cocycle $q$ of degree $n$ 
     on $(X,\trid)$, let ${\Bbbk}X$ denote the formal linear span of elements in $X$. Then
     \begin{align}
     c_q\colon (\Bbbk X\otimes \Bbbk^n)\otimes (\Bbbk X\otimes \Bbbk^n)&\to (\Bbbk X\otimes \Bbbk^n)\otimes (\Bbbk X\otimes \Bbbk^n)\\
(x\otimes v)\otimes (y\otimes w)&\mapsto (x\trid y\otimes q({x,y})(w))\otimes (x\otimes v)
\end{align}
is a solution of the braid equation \eqref{eq:braid1} for $A=B=C=\Bbbk X$. We can view $\Bbbk X$ as an object in the additive $\Bbbk$-linear braided monoidal category $\Vu$ generated by $\Bbbk X$ and $c_q$. Conversely, every solution of \eqref{eq:braid1} stemming from a finite-dimensional Yetter-Drinfeld module over a finite group can be constructed from a rack and a cocycle, \cite[Thm 1.14]{AG}. 

\medskip

We recall the construction of the Nichols algebra from Example \ref{ex:tensor-cotensor-Nichols} in this setup, which is valid for any vector space $V$ equipped with a solution $c=c_{V,V}$ of \eqref{eq:braid1}. For $n\geq2$, we denote the symmetric group on $n$ letters by $\Sn$ and choose $s_i:=(i\ i+1)\in \Sn$ for $i=1,\,\ldots,\,n-1$ as a set of Coxeter generators. Let $B_n$ be the braid group  with presentation: 
\[
\langle \sigma_1, \dots, \sigma_{n-1} \; | \; \sigma_i \sigma_{i+1} \sigma_i = \sigma_{i+1} \sigma_i \sigma_{i+1}, \; \sigma_i \sigma_j =\sigma_j \sigma_i, \mbox{if $|i-j|>1$} \rangle.
\]
The {Matsumoto section} is the map $M_n\colon\Sn\to B_n$ determined uniquely by the condition $M_n(s_{i_1}\cdots s_{i_l})=\sigma_{i_1}\cdots\sigma_{i_l}$ for any reduced expression $s_{i_1}\cdots s_{i_l}$. 

For $n\geq 2$ the action of the braid group $B_n$  associated with $V$ and $c$
is the representation $\rho_n\colon B_n\to\GL(V^{\otimes n})$ obtained by letting $\sigma_i$ act as $c$ on the $i$-th and $(i+1)$-th tensor factors and trivially on the others. The $n$-th {quantum symmetrizer} $Q_n$ is then defined as the endomorphism $\sum_{\sigma\in\Sn}\rho_n(M_n(\sigma))$ of $V^{\otimes n}$.

The {Nichols algebra} ${\mathcal B}(V,c)$ is the graded algebra $T(V)/\bigoplus_{n\geq 2}{\rm ker}(Q_n)$ with coproduct determined by requiring the elements in  $V$ to be primitive. 

\medskip

 Let $(C,\trid)$ be a rack given by a union of conjugacy classes in a finite group $G$. Let $q,\, q'$ be $2$-cocycles of degree 1  on $(C,\trid)$. Then  $q,\, q'$   are said to be {twist-equivalent} if  there exists  a $2$-cocycle $\sigma$ on $G\times G$ satisfying 
    \begin{equation}\label{eq:rtwist}q'({x,y})=\sigma(x,y)q({x,y})\sigma(x\trid y,x)^{-1},\quad \mbox{ for all }x,\,y\in C.\end{equation}

\begin{corollary}\label{cor:d-twist}
Let $\Bbbk X$ be a Yetter-Drinfeld module of a finite group $G$, supported by a conjugacy class and let $q$ be the corresponding cocycle. Let $q'$ be a a twist-equivalent cocycle $X$, and let ${\mathcal B}(X,q)$ and ${\mathcal B}(X,q')$ be the Nichols algebras associated with  $(\Bbbk X,c_q)$ and $(\Bbbk X,c_{q'})$, respectively.
Then, for $d\in\NN$ there holds ${\mathcal B}(X,q)\simeq\widehat{{\mathcal B}(X,q)}_d$ if and only if  ${\mathcal B}(X,q')\simeq\widehat{{\mathcal B}(X,q')}_d$. 
\end{corollary}
\begin{proof} Let $\sigma\in Z^2(\Bbbk G,\Bbbk^\times)$ satisfy \eqref{eq:rtwist} for $q$ and $q'$. It induces a monoidal autoequivalence of categories from $\,^{\Bbbk G}_{\Bbbk G}\mathcal{YD}$ to $\,^{\Bbbk G}_{\Bbbk G}\mathcal{YD}$. According to \cite[Lemma 3.9]{AFGV}, the vector space $\Bbbk X$ with braiding $c_{q'}$ is the image through $\mathcal T_\sigma$ of $\Bbbk X$. The statement follows from Theorems \ref{Thm1} and \ref{thm:twistapprox} and functoriality of the Nichols algebra.  
\end{proof}

\appendix
\section{An alternative proof of Corollary \ref{cor:d-twist}}
\label{Appendix}
In this appendix we present an alternative proof of Corollary \ref{cor:d-twist}. As a stepping stone, we prove the following proposition which is a direct consequence of the arguments in \cite[Theorem 2.7]{MO} and \cite[Theorem 3.8]{AFGV}.  We spell out a direct proof that does not require any categorical language. 
\begin{proposition}\label{prop:Bnrep-iso}Let $X$ be a union of conjugacy classes in a finite group $G$. If $q$ and $q'$ are twist-equivalent cocycles on $X$ with values in $\Bbbk^\times$, then for any $n\geq 2$ the $B_n$-representations associated respectively with $(\Bbbk X,c_{q'})$ and $(\Bbbk X,c_{q})$, are isomorphic.
\end{proposition}
\begin{proof}
Let $\sigma\in Z^2(\Bbbk G,\Bbbk^\times)$ satisfy \eqref{eq:rtwist} for $q$ and $q'$. For $n\geq2$, let $\rho'_n$ and $\rho_n$ be the  $B_n$-representations associated with $(\Bbbk X,c_{q'})$ and $(\Bbbk X,c_{q})$, respectively.
We define the linear isomorphism 
\begin{align*}
    f\colon(\Bbbk X)^{\otimes n}&\to (\Bbbk X)^{\otimes n}\\
    x_1\otimes\cdots\otimes x_n&\mapsto \left(\prod_{i=1}^{n-1}\sigma(x_i,x_{i+1}\cdots x_n) \right)x_1\otimes\cdots\otimes x_n
    \end{align*}
    for $x_i\in X$ with $i=1,\,\ldots,\,n$ and we show that
    $\rho_n(\sigma_j)f=f( \rho'_n(\sigma_j))$ for any generator $\sigma_j\in B_n$. 
    Let $x_1,\,\ldots,\,x_n\in X$. We set
    $\underline{x}=x_1\otimes\cdots \otimes x_{j-1}\otimes (x_j\trid x_{j+1})\otimes x_j\otimes x_{j+2}\otimes\cdots\otimes x_n$ and ${\mathbf x}=x_{j+2}\cdots x_n$, (product in $G$). Then
    \begin{align*}
    \rho_n(\sigma_j)f(x_1\otimes\cdots\otimes x_n)=\left(\prod_{i=1}^{n-1}\sigma(x_i,x_{i+1}\cdots x_n) \right)q({x_j,x_{j+1}})\underline{x}
\end{align*}
and
\begin{align*}
   f(\rho'_n(\sigma_j))(x_1\otimes\cdots\otimes x_n)=q'({x_j,x_{j+1}})\Bigl(\prod_{\genfrac{}{}{0pt}{2}{1\leq i\leq j-1\mbox{ or }}{
   j+2\leq i\leq n-1}}\sigma(x_i,x_{i+1}\cdots x_n) \Bigr)
   \sigma(x_j\trid x_{j+1},x_j\mathbf{x})\sigma(x_j,{\mathbf x})
   \underline{x}.
\end{align*}
   These expressions are equal if and only if 
   \begin{equation}\label{eq:condition}\sigma(x_j,x_{j+1}\mathbf x)\sigma(x_{j+1},\mathbf{x})\sigma(x_j\trid x_{j+1},x_j)=\sigma(x_j\trid x_{j+1},x_j\mathbf x)\sigma(x_j,x_{j+1})\sigma(x_j,\mathbf x).\end{equation}
   Applying \eqref{eq:cocycle} to $\{x_j, x_{j+1},\mathbf x\}$ on the left hand side of \eqref{eq:condition} and to $\{x_j\trid x_{j+1}, x_j, \mathbf x\}$ on the right hand side we obtain $\sigma(x_j,x_{j+1})\sigma(x_j x_{j+1},\mathbf x)\sigma(x_j\trid x_{j+1},x_j)$ on both sides.
\end{proof}

\begin{proof}[Proof of Corollary \ref{cor:d-twist}]
For $n\geq 2$, let $Q_n$ be the quantum symmetrizer corresponding to the $B_n$-representation $\rho_n$ associated with $(X,q)$. Now $\widehat{{\mathcal B}(X,q)}_d\simeq {\mathcal B}(X,q)$ if and only if 
the natural algebra projection $\widehat{{\mathcal B}(X,q)}_d\to {\mathcal B}(X,q)$ is a linear isomorphism, that is, if and only if for each $n>d$ the vector spaces $(\Bbbk X)^{\otimes n}/{\rm ker}\; (Q_n)$ and $(\Bbbk X)^{\otimes n}/\left(\bigoplus_{j=1}^d{\rm ker}\; (Q_j)\right)\cap (\Bbbk X)^{\otimes n}$ have the same dimension. Here, $\left(\bigoplus_{j=1}^d{\rm ker}\;(Q_j)\right)$ denotes the ideal generated by $\bigoplus_{j=1}^d{\rm ker}\;Q_j$ in $T(\Bbbk X)$. The dimension of the above quotients depends only on the isomorphism class of $\rho_n$ for any $n$. 
The claim follows from Proposition \ref{prop:Bnrep-iso}.
\end{proof}

\medskip

{\bf Funding and Acknowledgement.} The authors are extremely grateful to the referee for the very careful reading of the manuscript.
Project funded by the EuropeanUnion – NextGenerationEU under the National
Recovery and Resilience Plan (NRRP), Mission 4 Component 2 Investment 1.1 -
Call PRIN 2022 No. 104 of February 2, 2022 of Italian Ministry of
University and Research; Project 2022S8SSW2 (subject area: PE - Physical
Sciences and Engineering)``Algebraic and geometric aspects of Lie theory".  The authors participate in the INdAM group GNSAGA.


\begin{thebibliography}{50}

\bibitem{A2} N. Andruskiewitsch, 
\textit{An introduction to Nichols algebras}. Quantization, geometry and noncommutative structures in mathematics and physics, 135--195. Math. Phys. Stud. Springer, Cham, 2017.

 \bibitem{AFGV} N. Andruskiewitsch, F. Fantino, G. A. Garc\'ia, L. Vendramin, \textit{On Nichols algebras associated to simple racks}, Groups, algebras and applications, 
Contemp. Math., 537, American Mathematical Society,  2011, 31--56.

 \bibitem{AG} N. Andruskiewitsch, M. Gra\~{n}a,  \textit{From racks to pointed Hopf algebras}, Adv. Math. {\bf 178} (2003), no.2, 177--243.

\bibitem{CERyD} Carnovale, G.; Esposito, F.; Rubio y Degrassi, L. Truncated factorized perverse sheaves on Sym(C), arXiv:2502.21213

  
\bibitem{C}{K. Coulembier},  Tensor ideals, Deligne categories and invariant theory.  Selecta Math. (N.S.) {\bf 24} (2018), no. 5, 4659--4710. 

 


\bibitem{Doi} Y. Doi, \textit{Braided bialgebras and quadratic bialgebras}. Comm. Algebra {\bf 21} (1993), no.5, 1731–1749.




\bibitem{EGNO} P. Etingof, S. Gelaki,  D. Nykshyck, V. Ostrik, \textit{Tensor categories}, Math. Surveys Monogr., 205, American Mathematical Society, 2015.


   
\bibitem{G} J. W. Gray, \textit{Formal Category Theory}, Lecture Notes in Math. 391, Springer-Verlag, 1974.



\bibitem{HS} I. Heckenberger, H. J. Schneider, \textit{Hopf algebras and root systems}. Math. Surveys Monogr., 247 American Mathematical Society, 2020. 

\bibitem{Hopf} H. Hopf, \textit{\"Uber die Topologie der Gruppen-Mannigfaltigkeiten und ihre Verallgemeinerungen}, Ann. of Math. 42 (1941), 22-52. Reprinted in Selecta Heinz Hopf, pp. 119-151, Springer, Berlin (1964).

\bibitem{humphreys} J. Humphreys, Introduction to Lie Algebras and Representation Theory, GTM 9, Springerr, (1972).



\bibitem{KS3} M.  Kapranov,  V. Schechtman,  \textit{Shuffle algebras and perverse sheaves},  Pure and Applied Mathematics Quarterly,  Special Issue in Honor of Prof. Kyoji Saito's 75th Birthday vol 16,  573--657 (2020), International Press.

\bibitem{MCL} S. Mac Lane, \textit{Categories for the Working Mathematician,}   New York: Springer-Verlag (1978).

\bibitem{MO} S. Majid, R. Oeckl, \textit{Twisting of quantum differentials and the Planck scale Hopf algebra}, Comm. Math. Phys. {\bf 205} (1999), no. 3, 617--655.
 
\bibitem{Mc} J. McCleary,  \textit{A user's guide to spectral sequences}. Cambridge Stud. Adv. Math., 58
Cambridge University Press, Cambridge, 2001.

\bibitem{MM}J. W. Milnor, J. C. Moore, On the structure of Hopf algebras, Ann. of Math. (2), 81,  211--264, (1965).

 \bibitem {MimTod} M. Mimura,  H. Toda, \textit{Topology of Lie Groups, I and II}. Amer. Math. Soc., 1991.


\bibitem{R}{D. Radford,} \textit{Hopf Algebras, Series on Knots and Everything}, Vol. 49, World Scientific, 2012.


\bibitem{Ree} R. Ree, Lie Elements and an Algebra Associated With Shuffles, Ann. of Math.,  68(2) 210--220 (1958).

\bibitem{Sch}{P. Schauenburg}, \textit{A characterization of the Borel-like subalgebras of quantum enveloping algebras}, Comm. Algebra {\bf 24} (9), 2811--2823 (1996). 

\bibitem{scha}{P. Schauenburg}, \textit{On the braiding of a Hopf algebra in a braided category}, New York J. Math.  {\bf 4}, 259--263 (1998). 

\bibitem{Schn}{O. M. Schn\"urer}, \textit{Equivariant sheaves on flag varieties}, Math. Zeitschrift {\bf 267} 1-2, (2011)
27--80.



\end{thebibliography}
\end{document}